\documentclass[11pt]{amsart}
\usepackage{amsthm}
\usepackage{amscd}  
\usepackage{graphicx}
\usepackage{amssymb}
\usepackage{epstopdf}

\theoremstyle{plain}
\newtheorem{Thm}{Theorem}[section]

\newtheorem{Prop}[Thm]{Proposition}
\newtheorem{Cor}[Thm]{Corollary}
\newtheorem{Lem}[Thm]{Lemma}

\theoremstyle{definition}
\newtheorem{Defn}[Thm]{Definition}

\newtheorem{Rem}[Thm]{Remark}
 
\numberwithin{equation}{section}

\title{Non-Commutative Deformations of Derived McKay Correspondence for $A_n$ singularities}
\author{Yujiro Kawamata}
\date{}                                           
\usepackage{fancyhdr}
\begin{document}
\maketitle
\tableofcontents

\begin{abstract}
Let $Y = \mathbf C^2/\mathbf Z_{n+1}$ be a surface singularity of type $A_n$, let $f: X \to Y$ be the minimal resolution, 
and let $S = k[u,v] \# \mathbf Z_{n+1}$ be 
a twisted group algebra.
The derived McKay correspondence says that there is an equivalence of triangulated categories 
$D^b(\text{coh}(X)) \cong D^b(\text{mod-}S)$ over $Y$.
We will prove that this derived equivalence extends between the semi-universal non-commutative deformations 
of the commutative crepant resolution $X$ and the non-commutative crepant resolution $S$. 
 
\end{abstract}

\section{Introduction}

The purpose of this paper is to show that the derived McKay correspondence extends under non-commutative (NC)
deformations in an example.
This kind of extension is expected naturally because the Fourier-Mukai equivalence extends under NC deformations
in some examples (\cite{Toda}, \cite{BBBP}) and there is \lq\lq Mukai implies McKay'' principle (\cite{BKR}).

The derived McKay correspondence conjecture is a relationship between commutative and NC crepant 
resolutions.
Let $Y$ be an algebraic variety with only Gorenstein singularities.
A proper birational morphism $f: X \to Y$ from a smooth algebraic variety is said to be {\em crepant}
if there is no discrepancy between the canonical divisors, i.e., $f^*K_Y = K_X$.
Such $f$ is called a {\em commutative crepant resolution} (CCR) in this paper.
The non-commutative counterpart of CCR is a {\em non-commutative crepant resolution} (NCCR)
which is defined as follows by \cite{VdBergh} (also \cite{VdBergh2}): 
it is a coherent sheaf of NC associative $\mathcal O_Y$-algebras $S$ 
which is Cohen-Macaulay, 
has finite global dimension, and has a form $S = \text{End}_Y(M)$ for a reflexive $\mathcal O_Y$-module $M$.
We note that CCR or NCCR may or may not exist, and they are not necessarily unique when they exist.
There are also mixed cases of CCR and NCCR (cf. \cite{DK}).

The derived McKay correspondence conjecture claims that, if there exist both CCR $X$ and NCCR $S$, 
then there is an $\mathcal O_Y$-linear equivalence of triangulated categories 
$D^b(\text{coh}(X)) \cong D^b(\text{coh}(S))$.
We note that the underlying abelian categories $\text{coh}(X)$ and $\text{coh}(S)$ are completely different.
The former is geometric while the latter is algebraic.
For example, the set of simple objects over a point sheaf $\mathcal O_y$ for $y \in \text{Sing }Y$ is infinite
in $\text{coh}(X)$ but finite in $\text{coh}(S)$.
There are positive answers to this conjecture in special cases, e.g., 
in lower dimensional cases (\cite{KV} in dimension 2 and \cite{BKR} in dimension 3) and toric cases (\cite{toric}).
There are generalizations of the derived McKay correspondence to non-Gorenstein singularities, e.g., KLT singularities,
where the equivalences are replaced by fully faithful functors (cf. \cite{DK}, \cite{SDG}).

We consider in this paper the following question: when there is a derived McKay correspondence, is it extended
under NC deformations of CCR and NCCR?
A similar question was considered by \cite{Toda} and \cite{BBBP} for Fourier-Mukai equivalences.
According to \cite{BKR}, there is a principle that Mukai implies McKay, 
so it is natural to ask a similar question for McKay correspondences.
We will confirm this expectation in an explicit example by calculation, 
namely we treat $2$-dimensional toric singularities.

We consider the following case.
Let $Y = k^2/G$ be a $2$-dimensional toric singularity, 
where $k = \mathbf C$ and $G = \mathbf Z/(n+1)$ for a positive integer $n$,
and the action of $G$ is given by $g(u,v) = (\zeta u, \zeta^{-1}v)$ for coordinates $(u,v)$ and 
$\zeta = e^{2\pi i/(n+1)}$.
It is an isolated singularity of type $A_n$.
CCR is the minimal resolution $f: X \to Y$, and 
NCCR is given by a twisted group algebra $S = k[u,v]\# G$, where the multiplication is defined by the rule
$(ag)(bh) = ag(b)gh$ for $a,b \in k[u,v]$ and $g,h \in G$.

According to \cite{CBH}, we consider the semi-universal NC deformation of $S$ 
given by $\mathcal S = k[s_0,\dots,s_n]\langle u,v \rangle \# G/(uv - vu - \sum_{i=0}^n s_ig^i)$, 
where $s_0,\dots,s_n$ are algebraically independent deformation parameters.
It is a flat deformation of $S$ over an affine space.

The NC deformation $\mathcal X$ of $X$ we consider in this paper is a kind of an NC scheme.
It is constructed by gluing NC affine rings in the same way as $X$ is covered by commutative affine rings.
Here we note that the localization is not possible in general for NC rings, so we do not consider a space
with a structure sheaf of rings.
Instead an NC scheme $\mathcal X$ is defined as a collection of NC rings with gluing homomorphisms, 
which satisfy certain conditions (cocycle condition, flatness, and the birationality). 
$\mathcal X$ is flat over another polynomial ring $k[t_0,\dots,t_n]$.
We note that the base ring is global, and it is not infinitesimal nor formal unlike the deformation quantization. 

We then define the bounded derived categories of coherent modules on NC deformations of CCR and NCCR.
The following is the main result:

\begin{Thm}\label{main}
Let $Y$ be a $2$-dimensional toric singularity of type $A_n$, let $X$ and $S$ be its CCR and NCCR, 
and let $\mathcal X$ and $\mathcal S$ be their semi-universal NC deformations over polynomial algebras
$k[t_0,\dots,t_n]$ and $k[s_0,\dots,s_n]$ respectively.
Then there is an equivalence of triangulated categories
\[
D^b(\text{coh}(\mathcal X)) \cong D^b(\text{coh}(\mathcal S))
\]
which is linear over a linear coordinate change $(t_0,\dots,t_n) \to (s_0,\dots,s_n)$.
\end{Thm}

The right hand side is the usual bounded derived category of finitely generated $\mathcal S$-modules.
Since $S$ is an NC algebra, we can write $D^b(\text{mod-}\mathcal S)$ instead of 
$D^b(\text{coh}(\mathcal S))$.
For the left hand side, we define an $\mathcal X$-module as a collection of modules over the NC rings 
with the compatibility isomorphisms.
We will prove the theorem by showing that there is a tilting object in the category of coherent $\mathcal X$-modules, 
and its endomorphism ring is isomorphic to $\mathcal S$.

\vskip 1pc

The outline of the paper is as follows.
In \S 2, we define NC schemes and modules on them.
We note that such NC schemes are already considered in \cite{DVLL} in full generality.
Our theorem partly justifies our definition of NC schemes. 

In \S 3, we construct NC deformations of CCR and NCCR for our toric surface singularities.
The case for NCCR is the algebra $\mathcal S$ which is already constructed in \cite{CBH}, 
and the case for CCR is our new NC scheme $\mathcal X$.

We investigate the category of quasi-coherent modules on our $\mathcal X$ in \S 4.
It is a Grothendieck abelian category, and we can consider cohomology groups based on injective resolutions.
But they are difficult to calculate, so we will prove that they are isomorphic to \v Cech cohomology groups in the 
case which is needed.
We will prove that there is a locally free coherent module $\mathcal T$ which is tilting, 
i.e., there are no higher cohomology groups.
We also prove that $\mathcal T$ generates the derived category of coherent sheaves.  

We will compare the endomorphism algebra $\text{End}(\mathcal T)$ and the NC deformed algebra $\mathcal S$, 
and prove that they are isomorphic in \S 5.
We will prove it by showing that they are both quiver algebras whose bases are of type $\tilde A_n$, and the
sets of generators and relations are the same.
Then we conclude the proof of Theorem \ref{main} by the tilting theory. 

\vskip 1pc

The author would like to thank NCTS of National Taiwan University where 
the work was partly done while the author visited there.
This work is partly supported by JSPS Kakenhi 21H00970.

\section{NC schemes and (quasi-)coherent sheaves}
 
We will give temporary definitions of NC schemes and (quasi-)coherent sheaves on them which will be used in this paper.

\begin{Defn}
Let $I$ be a poset.
An {\em NC scheme} $X$ parametrized by $I$ consists of associative NC algebras $R_i$ for $i \in I$
and algebra homomorphisms $\phi_{ji}^X: R_i \to R_j$ for all $i < j$, called the {\em gluing homomorphisms}, 
which satisfy the following conditions:

(1) (cocycle condition) $\phi_{kj}^X \circ \phi_{ji}^X = \phi_{ki}^X$ for $i < j < k$. 

(2) (flatness) $R_j$ is flat as left and right $R_i$-modules through $\phi_{ji}^X$.

(3) (birationality condition) The natural homorphism of $R_j$-modules 
$R_j \otimes_{R_i} R_j \to R_j$ is bijective for $i < j$.
\end{Defn}

We note that the existence of the maximum $\max \{i,j\} \in I$ for $i,j \in I$ is not required in this paper.

For example, let $X$ be a (usual) algebraic variety 
and let $U_i = \text{Spec }R_i$ for $i \in I$ be its open covering by affine open subsets.
We define a partial order on $I$ as follows: $i < j$ if $U_i \supset U_j$.
Then the homomorphism $\phi_{ji}^X: R_i \to R_j$ is defined by the restrictions of functions.
Since $U_i$ and $U_j$ are birationally equivalent, the condition (3) holds.

\begin{Defn}
An {\em NC deformation} of an NC scheme $X = \{R_i\}_{i \in I}$ over a (commutative) local ring $(B, M)$ 
is a pair $(\mathcal X, \alpha)$ consisting of 
an NC scheme $\mathcal X = \{\mathcal R_i\}_{i \in I}$ such that the algebras $\mathcal R_i$ are flat $B$-algebras for all $i$, 
and isomorphisms $\alpha_i: B/M \otimes_B \mathcal R_i \cong R_i$ for all $i$, 
which are compatible with the gluing homomorphisms $\phi_{ji}^{\mathcal X}$ and $\phi_{ji}^X$.
\end{Defn}

\begin{Defn}
A {\em quasi-coherent} module $M = (M_i, \psi_{ji}^M)$ on an NC scheme $X = \{R_i\}$ consists of 
right $R_i$-modules $M_i$ and  
gluing isomorphisms 
\[
\psi_{ji}^M: M_i \otimes_{R_i} R_j \to M_j
\]
as $R_j$-modules for $i < j$ which satisfy the 
cocycle condition $\psi^M_{kj} \circ (\psi^M_{ji} \otimes_{R_j} \text{Id}_{R_k})= \psi^M_{ki}$ for $i < j < k$.
If the $M_i$ are finitely generated, then $M$ is called {\em coherent}.

A homomorphism of quasi-coherent modules $h: M \to N$
is a collection of $R_i$-homomorphisms $h_i: M_i \to N_i$
which are compatible with gluing isomorphisms $\psi_{ji}^M$ and $\psi_{ji}^N$.
\end{Defn}

\begin{Lem}
(1) Let $h: M \to N$ be a homomorphism of quasi-coherent modules.
Then $\text{Ker}(h) = (\text{Ker}(h_i))$ and $\text{Coker}(h) = (\text{Coker}(h_i))$ are also quasi-coherent modules.

(2) The category of quasi-coherent modules $\text{Qcoh}(X)$ is an abelian category, and  
the category of coherent modules $\text{coh}(X)$ is an abelian subcategory. 
\end{Lem}

\begin{proof}
By the flatness of $\phi_{ji}: R_i \to R_j$, we have isomorphisms
$\text{Ker}(h_i) \otimes_{R_i} R_j \cong \text{Ker}(h_j)$ and 
$\text{Coker}(h_i) \otimes_{R_i} R_j \cong \text{Coker}(h_j)$.
Therefore we have our assertions.
\end{proof}

\section{NC deformation of CCR and NCCR for $A_n$}

Let $Y$ be a $2$-dimensional singularity of type $A_n$.
$Y$ has a commutative crepant resolution (CCR) $X$ and a non-commutative crepant resolution (NCCR) $S$.
We will describe NC deformations $\mathcal X$ and $\mathcal S$ of $X$ and $S$, respectively.

We consider a quotient variety $Y = (\text{Spec }k[u,v])/G$, where a cyclic group $G = \mathbf Z/(n+1)$ 
acts on an affine plane by $(u,v) \mapsto (\zeta u, \zeta^{-1}v)$ with
$\zeta = e^{2\pi i/(n+1)}$.
$Y$ is a toric variety with an isolated singularity of type $A_n$, and the minimal resolution $f: X \to Y$ is a toric morphism.
$X$ is a union of affine planes $U_i = \text{Spec }R_i \cong k^2$ with $R_i = k[x_i,y_i]$ ($0 \le i \le n$),
which are glued together by the following transformation rules:
\[
x_{i+1} = x_i^2y_i, \,\,\, y_{i+1} = x_i^{-1}.
\]

We deform these affine planes to non-commutative planes, and then glue them to obtain an NC variety $\mathcal X$.
The sheaf of bivectors $\bigwedge^2 T_X \cong \mathcal O_X$ has a nowhere vanishing section,
and $\mathcal X$ contains the corresponding purely non-commutative deformation.
Since $H^2(X, \mathcal O_X) \cong 0$, there is no twisted deformation (cf. \cite{NCdef}).

We define a poset 
\[
I = \{0, 1, \dots, n\} \cup \{(0,1), (1,2), \dots, (n-1,n)\}
\]
with the order $i, i+1 < (i,i+1)$.
The corresponding affine open subsets are the $U_i$ ($0 \le i \le n$) and the 
$U_{i,i+1} = \text{Spec }R_{i,i+1} \cong k^* \times k$ with $R_{i,i+1} = k[x_i,x_i^{-1},y_i]$ ($0 \le i < n$).
We consider only intersections $U_{i,i+1} = U_i \cap U_{i+1}$ and ignore other intersections.

We define deformed affine rings $\mathcal R_i$ by
\[
\mathcal R_i = k[t_0,\dots, t_n] \langle x_i,y_i \rangle/(x_iy_i - y_ix_i - t_0)
\]
where $t_0,\dots,t_n$ are deformation parameters, 
and define an NC variety $\mathcal X$ by gluing these NC rings
with the deformed transformation rules:
\[
x_{i+1} = x_i^2y_i + t_{i+1}x_i, \,\,\, y_{i+1} = x_i^{-1}
\]
for $0 \le i < n$.
The parameters $t_1, \dots, t_n$ correspond to the commutative deformations of $X$ constructed by changing the 
gluing, and $t_0$ corresponds to the purely non-commutative deformation.

More precisely, we define
\[
\begin{split}
&\mathcal R_{i,i+1} = k[t_0,\dots, t_n] \langle x_i, x_i^{-1}, y_i \rangle/(x_iy_i-y_ix_i-t_0) \\
&:= k[t_0,\dots, t_n] \langle x_i, y_i, y_{i+1} \rangle/(x_iy_i-y_ix_i-t_0, x_iy_{i+1}-1, y_{i+1}x_i - 1).
\end{split}
\]
The gluing is given by the natural homomorphisms
\[
\phi_i^+: \mathcal R_i \to \mathcal R_{i,i+1}, \,\,\,
\phi_i^-: \mathcal R_{i+1} \to \mathcal R_{i,i+1}.
\]

We note that the commutativity relations of the $x_i,y_i$ are compatible with gluing.
Indeed we have
\[
\begin{split}
&x_{i+1}y_{i+1}-y_{i+1}x_{i+1} = x_i^2y_ix_i^{-1} + t_{i+1} - (x_iy_i - t_{i+1}) \\
&= x_i(y_ix_i + t_0)x_i^{-1} - x_iy_i = t_0.
\end{split}
\]

\begin{Lem}
(1) $\mathcal R_i$ is Noetherian.

(2) $\phi_i^+$, $\phi_i^-$ are injective.

(3) $y_iy_{i+1} - y_{i+1}y_i = t_0y_{i+1}^2$.

(4) $\mathcal R_{i,i+1} = \bigcup_{l \ge 0} \mathcal R_i y_{i+1}^l$.

(5) Let $\mathcal M_i$ be a right $\mathcal R_i$-module.
Then
\[
\begin{split}
&\text{Ker}(\mathcal M_i \to \mathcal M_i \otimes_{\mathcal R_i} \mathcal R_{i,i+1}) 
= \mathcal M_{i, x_i\text{-tor}} := \{m \in \mathcal M_i \mid \exists l \ge 0, mx_i^l = 0\}, \\
&\mathcal M_i \otimes_{\mathcal R_i} \mathcal R_{i,i+1} 
= \bigcup_{l \ge 0} (\mathcal M_i/\mathcal M_{i, x_i\text{-tor}}) y_{i+1}^l.  
\end{split}
\]
 
(6) $\phi_i^+$, $\phi_i^-$ are flat.

(7) The birationality conditions hold:
\[
\mathcal R_{i,i+1} \otimes_{\mathcal R_i} \mathcal R_{i,i+1} \cong \mathcal R_{i,i+1}, \,\,\, 
\mathcal R_{i,i+1} \otimes_{\mathcal R_{i+1}} \mathcal R_{i,i+1} \cong \mathcal R_{i,i+1}.
\]
\end{Lem}

\begin{proof}
(1) We define degrees on $\mathcal R_i$ by
$\text{deg}(x_i,y_i,t_j) = (1,1,0)$, and 
define a filtration by $F_d(\mathcal R_i) = \{r \in \mathcal R_i \mid \text{deg}(r) \le d\}$.
Then we have $\text{Gr}^F(\mathcal R_i) \cong k[t_0,\dots, t_n][x_i,y_i]$, hence $\mathcal R_i$ is Noetherian.

(2) Since $\text{Gr}(\phi_i^+): \text{Gr}(\mathcal R_i) \to \text{Gr}(\mathcal R_{i,i+1}) \cong k[t_0,\dots, t_n][x_i,y_i, x_i^{-1}]$ is injective, so is 
$\phi_i^+$.

(3) This follows from $x_i(y_iy_{i+1} - y_{i+1}y_i)x_i = x_iy_i - y_ix_i = t_0$.

(4) and (5) are corollaries of (3).

(6) Let $\mathcal M_i \to \mathcal N_i$ be an injective homomorphism of $\mathcal R_i$-modules.
Then $\mathcal M_i/\mathcal M_{i, x_i\text{-tor}} \to \mathcal N_i/\mathcal N_{i, x_i\text{-tor}}$ is injective,
hence so is 
$\mathcal M_i \otimes_{\mathcal R_i} \mathcal R_{i,i+1} \to \mathcal N_i \otimes_{\mathcal R_i} \mathcal R_{i,i+1}$.
Thus $\mathcal R_{i,i+1}$ is flat over $\mathcal R_i$.

(7) is a corollary of (5).
\end{proof}

Next we define an NCCR and its NC deformation for $A_n$ following \cite{CBH}.
We have $Y = \text{Spec }k[u,v]^G$, where $G = \mathbf Z/(n+1)$ acts by $(u,v) \mapsto (\zeta u, \zeta^{-1}v)$ with
$\zeta = e^{2\pi i/(n+1)}$.
A non-commutative crepant resolution (NCCR) of $Y$ is given by a {\em twisted group ring}
$S = k[u,v] \# G$, where the multiplication is defined as $ag \cdot bh = ag(b)gh$ for $a,b \in k[u,v]$ and $g,h \in G$.

According to \cite{CBH}, we define the \lq\lq versal'' NC deformation $\mathcal S$ of $S$ by
\[
\mathcal S = k[s_0,\dots, s_n]\langle u,v \rangle \# G/(uv - vu - \sum_{i=0}^n s_ig^i)
\]
where $s_0, \dots, s_n$ are deformation parameters and $g$ is a generator of $G$.
$s_1, \dots, s_n$ parametrize commutative deformations and $s_0$ corresponds to a purely non-commutative deformation.
Indeed, the ring of invariants $e \mathcal S e$ of $\mathcal S$ for $e = \sum_{i=0}^n g^i/(n+1)$
becomes commutative when we put $s_0 = 0$ (loc. cit.).

\section{$\text{Qcoh}(\mathcal X)$ for $A_n$}

We will investigate the structure of the category $\text{Qcoh}(\mathcal X)$ of quasi-coherent modules 
over the NC deformation $\mathcal X$ for the singular surface $Y$ of type $A_n$.

\subsection{\v Cech cohomology}

We will define \v Cech cohomology groups for NC schemes in this subsection.
They will be compared with the cohomology groups defined by injective resolutions later.
We consider here only $1$-dimensional special case.
The general case will be treated in a subsequent paper.  

We recall that a quasi-coherent module $\mathcal M$ on $\mathcal X$ is a collection of 
right $\mathcal R_i$-modules $\mathcal M_i$ and right $\mathcal R_{i,i+1}$-modules $\mathcal M_{i,i+1}$
together with isomorphisms
\[
\begin{split}
&\psi^{\mathcal M}_{i,+}: \mathcal M_i \otimes_{\mathcal R_i} \mathcal R_{i,i+1} \to \mathcal M_{i,i+1}, \\
&\psi^{\mathcal M}_{i,-}: \mathcal M_i \otimes_{\mathcal R_i} \mathcal R_{i-1,i} \to \mathcal M_{i-1,i}.
\end{split}
\]
A homomorphism $h: \mathcal M \to \mathcal N$ consists of $\mathcal R_i$-homomorphisms
$h_i: \mathcal M_i \to \mathcal N_i$ and $\mathcal R_{i,i+1}$-homomorphisms
$h_{i,i+1}: \mathcal M_{i,i+1} \to \mathcal N_{i,i+1}$ such that the following diagrams are commutative:
\[
\begin{CD}
\mathcal M_i \otimes_{\mathcal R_i} \mathcal R_{i,i+1} @>{\psi^{\mathcal M}_{i,+}}>> \mathcal M_{i,i+1} \\
@V{h_i \otimes_{\mathcal R_i} \mathcal R_{i,i+1}}VV @V{h_{i,i+1}}VV \\
\mathcal N_i \otimes_{\mathcal R_i} \mathcal R_{i,i+1} @>{\psi^{\mathcal N}_{i,+}}>> \mathcal N_{i,i+1}, \\
\mathcal M_i \otimes_{\mathcal R_i} \mathcal R_{i-1,i} @>{\psi^{\mathcal M}_{i,-}}>> \mathcal M_{i-1,i} \\
@V{h_i \otimes_{\mathcal R_i} \mathcal R_{i-1,i}}VV @V{h_{i-1,i}}VV \\
\mathcal N_i \otimes_{\mathcal R_i} \mathcal R_{i-1,i} @>{\psi^{\mathcal N}_{i,-}}>> \mathcal N_{i-1,i}.
\end{CD}
\]
The following is an immediate consequence of the definition:

\begin{Lem}
Let $\mathcal M,\mathcal N \in \text{Qcoh}(\mathcal X)$ be quasi-coherent modules.
Define a $k$-homomorphism 
\[
\begin{split}
\Delta: &\prod_{i=0}^n \text{Hom}_{\mathcal R_i}(\mathcal M_i, \mathcal N_i)
\to \prod_{i=1}^n \text{Hom}_{\mathcal R_{i-1,i}}(\mathcal M_{i-1,i}, \mathcal N_{i-1,i}) \\
&(h_0,\dots,h_n) \mapsto (g_{0,1}, \dots, g_{n-1,n})
\end{split}
\]
by 
\[
\begin{split}
g_{i-1,i} = &\psi^{\mathcal N}_{i-1,+} \circ (h_{i-1} \otimes_{\mathcal R_{i-1}} \mathcal R_{i-1,i}) \circ (\psi^{\mathcal M}_{i-1,+})^{-1}
\\ &- \psi^{\mathcal N}_{i,-} \circ (h_i \otimes_{\mathcal R_i} \mathcal R_{i-1,i}) \circ (\psi^{\mathcal M}_{i,-})^{-1}.
\end{split}
\]
Then
\[
\text{Hom}_{\mathcal X}(\mathcal M, \mathcal N) \cong \text{Ker}(\Delta).
\]
\end{Lem}

\begin{Defn}
We define the first \v Cech cohomology group by
\[
\text{\v CHom}^1_{\mathcal X}(\mathcal M, \mathcal N) 
\cong \text{Coker}(\Delta).
\]
\end{Defn}

\begin{Defn}
A locally free module  $\mathcal M$ of rank $1$ on $\mathcal X$ is a coherent module 
such that $\mathcal M_i \cong \mathcal R_i$ as right $\mathcal R_i$-modules for all $i$.
A {\em divisorial sheaf} $\mathcal D = \mathcal R(\sum_{i=1}^n d_iD_i)$ is a locally free module
of rank $1$ defined by gluing homomorphisms
\[
\psi_{i,-}^{\mathcal D}(r) = x_{i-1}^{d_i}r, \,\,\, \psi_{i,+}^{\mathcal D}(r) = r
\]
for $1 \le i \le n$ and $0 \le i < n$, respectively.
\end{Defn}

The notation \lq\lq $D_i$'' does not have an independent meaning, but it comes from a prime divisor 
on the minimal resolution $X$
which intersects transversally to the exceptional curve $C_i$. 

\begin{Lem}\label{LES}
Let $\mathcal D$ be a divisorial sheaf and let
$0 \to \mathcal N \to \mathcal N' \to \mathcal N'' \to 0$ be an exact sequence of quasi-coherent 
$\mathcal R$-modules.
Then there is a long exact sequence
\[
\begin{split}
&0 \to \text{Hom}_{\mathcal X}(\mathcal D, \mathcal N) \to \text{Hom}_{\mathcal X}(\mathcal D, \mathcal N') 
\to \text{Hom}_{\mathcal X}(\mathcal D, \mathcal N'') \\
&\to \text{\v CHom}^1_{\mathcal X}(\mathcal D, \mathcal N) \to \text{\v CHom}^1_{\mathcal X}(\mathcal D, \mathcal N') \to 
\text{\v CHom}^1_{\mathcal X}(\mathcal D, \mathcal N'') \to 0.
\end{split}
\]
\end{Lem}

\begin{proof}
Since $\mathcal D_i \cong \mathcal R_i$, we have an isomorphism as $k$-modules
$\text{Hom}_{\mathcal R_i}(\mathcal D_i, \mathcal N_i) \to \mathcal N_i$ given by
$h_i \mapsto h_i(1)$.
Hence we have the following commutative diagram of exact sequences:
\[
\begin{CD}
0 \to \prod_{i=0}^n (\mathcal D_i, \mathcal N_i) 
@>>> \prod_{i=0}^n (\mathcal D_i, \mathcal N'_i) 
@>>> \prod_{i=0}^n (\mathcal D_i, \mathcal N''_i) \to 0 \\
@VVV @VVV @VVV \\
0 \to \prod_{i=1}^n (\mathcal D_{i-1,i}, \mathcal N_{i-1,i})
@>>> \prod_{i=1}^n (\mathcal D_{i-1,i}, \mathcal N'_{i-1,i}) 
@>>> \prod_{i=1}^n (\mathcal D_{i-1,i}, \mathcal N''_{i-1,i}) \to 0
\end{CD}
\]
where we use abbreviations $(\mathcal D_i, \mathcal N_i) = \text{Hom}_{\mathcal R_i}(\mathcal D_i, \mathcal N_i)$ and
$(\mathcal D_{i-1,i}, \mathcal N_{i-1,i}) = 
\text{Hom}_{\mathcal R_{i-1,i}}(\mathcal D_{i-1,i}, \mathcal N_{i-1,i}))$, etc.
Then our assertion is a consequence of the snake lemma.
\end{proof}

\subsection{Generator}

We will prove in this subsection that $\text{Qcoh}(\mathcal X)$ is generated by an object 
$\bigoplus_{m_1,\dots, m_n \in \mathbf Z} \mathcal R(\sum_{i=1}^n m_iD_i)$.

\begin{Lem}\label{generator}
Let $\mathcal M$ be a non-zero quasi-coherent module.
Then there exist integers $d^0_i$ such that 
$\text{Hom}_{\mathcal X}(\mathcal R(- \sum_i d_iD_i), \mathcal M) \ne 0$ if $d_i \ge d^0_i$.
\end{Lem}

\begin{proof}
Since $\mathcal M \not\cong 0$, there is $i$ such that $\mathcal M_i \not\cong 0$.
Then there is a non-zero homomorphism $h_i: \mathcal R_i \to \mathcal M_i$.
We will extend $h_i$ to neighboring $\mathcal M_j$'s inductively.
We set $h_i(1) = m_i$.

We take sufficiently large positive integers $d^0_i$ and $d^0_{i+1}$ such that, 
for $d_i \ge d^0_i$ and $d_{i+1} \ge d^0_{i+1}$,  
\[
\begin{split}
&\psi_{i,-}^{\mathcal M}(m_i)x_{i-1}^{d_i} = m_{i-1} \otimes 1 
\in \text{Im}(\mathcal M_{i-1} \to \mathcal M_{i-1} \otimes_{\mathcal R_{i-1}} \mathcal R_{i-1,i}), \\
&\psi_{i,+}^{\mathcal M}(m_i)y_{i+1}^{d_{i+1}} = m_{i+1} \otimes 1 
\in \text{Im}(\mathcal M_{i+1} \to \mathcal M_{i+1} \otimes_{\mathcal R_{i+1}} \mathcal R_{i,i+1}) 
\end{split}
\] 
where we consider only $\psi_{i,\pm}^{\mathcal M}$ when $i+1 \le n$ or $i-1 \ge 0$, respectively.
Then we define homomorphisms 
\[
h_*: \mathcal R_*(-d_iD_i-d_{i+1}D_{i+1}) \to \mathcal M_*
\] 
for $* = i-1, (i-1,i), i, (i,i+1), i+1$ by
\[
\begin{split}
&h_{i-1}(r) = m_{i-1}r, \,\,\, h_i(r) = m_ir, \,\,\, h_{i+1}(r) = m_{i+1}r, \\
&h_{i-1,i}(r) = \psi_{i,-}^{\mathcal M}(m_ix_{i-1}^{d_i}r), \,\,\, h_{i,i+1}(r) = \psi_{i,+}^{\mathcal M}(m_ir).
\end{split}
\]
We check the compatibilities:
\[
\begin{split}
&\psi_{i-1.+}^{\mathcal M}h_{i-1}(r) = \psi_{i-1,+}^{\mathcal M}(m_{i-1}r) = \psi_{i,-}^{\mathcal M}(m_ix_{i-1}^{d_i}r) \\
&= h_{i-1,i}(r) = h_{i-1,i}\psi_{i-1,+}^{\mathcal D}(r), \\
&\psi_{i,-}^{\mathcal M}h_i(r) = \psi_{i,-}^{\mathcal M}(m_ir) = \psi_{i,-}^{\mathcal M}(m_ix_{i-1}^{d_i}x_{i-1}^{-d_i}r) \\
&= h_{i-1,i}(x_{i-1}^{-d_i}r) = h_{i-1,i}\psi_{i,-}^{\mathcal D}(r), \\
&\psi_{i,+}^{\mathcal M}h_i(r) = \psi_{i,+}^{\mathcal M}(m_ir) = h_{i,i+1}(r) = h_{i,i+1}\psi_{i,+}^{\mathcal D}(r), \\
&\psi_{i+1,-}^{\mathcal M}h_{i+1}(r) = \psi_{i+1,-}^{\mathcal M}(m_{i+1}r) = \psi_{i,+}^{\mathcal M}(m_iy_{i+1}^{d_{i+1}}r) \\
&= h_{i,i+1}(x_i^{-d_{i+1}}r) = h_{i,i+1}\psi_{i+1,-}^{\mathcal D}(r),
\end{split}
\]
where we denote $\mathcal D = \mathcal R(-d_iD_i-d_{i+1}D_{i+1})$.
Thus $h_i$ is extended to $h_{i-1}$, $h_{i+1}$, etc. in a compatible way.
Here we note that we use the same symbol $h_i$ to denote homomorphisms 
$\mathcal R_i \to \mathcal M_i$ and $\mathcal R_i(-d_iD_i-d_{i+1}D_{i+1}) \to \mathcal M_i$ by identifying 
$\mathcal R_i = \mathcal R_i(-d_iD_i-d_{i+1}D_{i+1})$.
We repeat this process to obtain the $h_j$ for all $j$.
\end{proof}

\begin{Lem}\label{generator1}
Let $\mathcal M$ be a coherent $\mathcal R$-module.

(1) Then there exist integers $d^0_i$ such that 
$\text{\v CHom}^1_{\mathcal X}(\mathcal R(- \sum_i d_iD_i), \mathcal M) = 0$ if $d_i \ge d^0_i$.

(2) There is a surjective homomorphism from a finite direct sum of divisorial sheaves to $\mathcal M$.
\end{Lem}

\begin{proof}
First we prove (1) in the case when $\mathcal M$ is itself a divisorial sheaf; 
we assume that $\mathcal M = \mathcal R(- \sum_i d'_iD_i)$ for some integers $d'_i$.
Then we have 
\[
\begin{split}
&\text{\v CHom}^1_{\mathcal X}(\mathcal R(- \sum_i d_iD_i), \mathcal R(- \sum_i d'_iD_i)) \\
&= \text{Coker}(\Delta: \prod_{i=0}^n \text{Hom}_{\mathcal R_i}(\mathcal R_i, \mathcal R_i)
\to \prod_{i=1}^n \text{Hom}_{\mathcal R_{i-1,i}}(\mathcal R_{i-1,i}, \mathcal R_{i-1,i})) \\
&= \text{Coker}(\Delta: \prod_{i=0}^n \mathcal R_i \to \prod_{i=1}^n \mathcal R_{i-1,i})
\end{split}
\]
since $\text{Hom}_{\mathcal R_i}(\mathcal R_i, \mathcal R_i) \cong \mathcal R_i$, etc., as $k$-modules, 
by sending $h_i: \mathcal R_i \to \mathcal R_i$ to $h_i(1)$.
The map $\Delta: (h_0, \dots, h_n) \mapsto (g_{0,1}, \dots, g_{n-1,n})$ is given by
\[
g_{i-1,i}(1) = h_{i-1}(1) - x_{i-1}^{-d'_i}h_i(1)x_{i-1}^{d_i}.
\]
Suppose that we are given the $g_{i-1,i}(1) \in \mathcal R_{i-1,i}$ for $1 \le i \le n$, and we look for 
the $h_i(1) \in \mathcal R_i$ for $0 \le i \le n$ which induce the $g_{i-1.i}$.
We proceed by induction on $i$.
We set $h_0(1) = 0$.
We take $d_1$ large enough so that 
$h_1(1) = - y_1^{-d'_1}g_{0,1}(1)y_1^{d_1} \in \mathcal R_1$.
Then we take $d_2$ large enough so that
$h_2(1) = - y_2^{-d'_2}(g_{1,2}(1) - h_1(1))y_2^{d_2} \in \mathcal R_2$.
In this way, we determine the $d_i$ and the $h_i$ to obtain our first claim.

\vskip 1pc

Now we take any coherent module $\mathcal M$ and prove (1) and (2) simultaneously.
We will construct a strictly increasing sequence of submodules $\{\mathcal M^t\}$ of $\mathcal M$ and 
a weakly increasing sequence of integers $\{d^t_i\}$ and $\{e^t_i\}$ such that the following hold:

(1t) $\text{\v CHom}^1_{\mathcal X}(\mathcal R(- \sum_i d_iD_i), \mathcal M^t) = 0$ if $d_i \ge d^t_i$, 

(2t) There are surjective homomorphisms
$\bigoplus_{j=1}^t \mathcal R(- \sum_i e^j_iD_i) \to \mathcal M^t$. 

\noindent
Since the $\mathcal R_i$ are Noetherian, there is a $t$ such that $\mathcal M^t = \mathcal M$, which is our claim.

We let $M^0 = 0$.
Assuming that we have already $\mathcal M^t$ and the $d^t_i$, we will construct $\mathcal M^{t+1}$ and the $d^{t+1}_i$.
We take large integers $e^t_i \ge d^t_i$ such that there is a non-zero homomorphism
$h'_t: \mathcal R(- \sum_i e^t_iD_i) \to \mathcal M/\mathcal M^t$.
By the induction assumption, we have 
$\text{\v CHom}^1_{\mathcal X}(\mathcal R(- \sum_i e^t_iD_i), \mathcal M^t) = 0$.
Hence there is a homomorphism 
$h_t: \mathcal R(- \sum_i e^t_iD_i) \to \mathcal M$ which induces $h'_t$.
We define $\mathcal M^{t+1}$ by $\text{Im}(h_t) = \mathcal M^{t+1}/\mathcal M^t$.
Then we have (2).

We take integers $d^{t+1}_i \ge e^t_i$ such that 
$\text{\v CHom}^1_{\mathcal X}(\mathcal R(- \sum_i d_iD_i), \mathcal R(- \sum_i e^t_iD_i)) = 0$ if $d_i \ge d^{t+1}_i$.
Then by Lemma \ref{LES}, we have 
$\text{\v CHom}^1_{\mathcal X}(\mathcal R(- \sum_i d_iD_i), \text{Im}(h_t)) = 0$, hence 
$\text{\v CHom}^1_{\mathcal X}(\mathcal R(- \sum_i d_iD_i), \mathcal M^{t+1}) = 0$ for $d_i \ge d^{t+1}_i$.
This is (1).
\end{proof}

\begin{Cor}
Any object in $\text{coh}(\mathcal R)$ has a locally free resolution whose terms are direct sums of 
divisorial sheaves.
\end{Cor}

The following is a special case of \cite{DVLL}:

\begin{Prop}
$\text{Qcoh}(\mathcal R)$ is a Grothendieck category (\cite{Grothendieck}).
\end{Prop}

\begin{proof}
We check the conditions.
The coproduct $\bigoplus_{\lambda \in \Lambda} \mathcal M_{\lambda}$ and the direct limit
$\varinjlim_{\lambda}\mathcal M_{\lambda}$ always exist.
For exact sequences $0 \to \mathcal L_{\lambda} \to \mathcal M_{\lambda} \to \mathcal N_{\lambda} \to 0$, 
we have an exact sequence
\[
0 \to \varinjlim_{\lambda} \mathcal L_{\lambda} \to \varinjlim_{\lambda}\mathcal M_{\lambda} \to 
\varinjlim_{\lambda}\mathcal N_{\lambda} \to 0.
\]
Finally, let $\mathcal G = \bigoplus_{m_1,\dots, m_n \in \mathbf Z} \mathcal R(\sum_{i=1}^n m_iD_i)$.
Then, for any object $\mathcal M \in \text{Qcoh}(\mathcal R)$, there exists an index set $\Lambda$ with a surjective homomorphism
$\mathcal G^{(\Lambda)} \to \mathcal M$.
\end{proof}

\begin{Cor} 
$\text{Qcoh}(\mathcal R)$ has enough injectives.
\end{Cor}

\subsection{Comparison of two cohomology groups}

Now that $\text{Qcoh}(\mathcal R)$ is a Grothendieck category, we can define 
the right derived functor $R\text{Hom}_{\mathcal X}(\mathcal M, \mathcal N)$ by using an injective resolution of 
$\mathcal N$.
For a divisorial sheaf $\mathcal D$, 
we will describe the cohomology groups $\text{Hom}^p_{\mathcal X}(\mathcal D, \mathcal N) = 
H^p(R\text{Hom}_{\mathcal X}(\mathcal D, \mathcal N))$ 
by using the C\v ech cohomology. 

\begin{Lem}\label{inj1}
Let $\mathcal D$ be a divisorial sheaf and let
$\mathcal N$ be a quasi-coherent module.
Then there is an injective homomorphism 
\[
\text{Hom}^1_{\mathcal X}(\mathcal D, \mathcal N) \to \text{\v CHom}^1_{\mathcal X}(\mathcal D, \mathcal N).
\]
Moreover, for a homomorphism of quasi-coherent modules $\mathcal N \to \mathcal N'$, 
there is a commutative diagram
\[
\begin{CD}
\text{Hom}^1_{\mathcal X}(\mathcal D, \mathcal N) @>>> \text{\v CHom}^1_{\mathcal X}(\mathcal D, \mathcal N)  \\
@VVV @VVV \\
\text{Hom}^1_{\mathcal X}(\mathcal D, \mathcal N') @>>> \text{\v CHom}^1_{\mathcal X}(\mathcal D, \mathcal N').
\end{CD}
\]
\end{Lem}

\begin{proof}
We embed $\mathcal N$ into an exact sequence $0 \to \mathcal N \to \mathcal I \to \mathcal L \to 0$ where
$\mathcal I$ is an injective module.
Then we have a commutative diagram of exact sequences
\[
\begin{CD}
(\mathcal D, \mathcal I) @>>> (\mathcal D, \mathcal L) @>>>
(\mathcal D, \mathcal N)^1 @>>> 0 \\
@V=VV @V=VV  \\
(\mathcal D, \mathcal I) @>>> (\mathcal D, \mathcal L) @>>>
\check C(\mathcal D, \mathcal N)^1 @>>> \check C(\mathcal D, \mathcal I)^1 
\end{CD}
\]
because $(\mathcal D, \mathcal I)^1 = 0$, 
where we abbreviate $(\mathcal D, \mathcal N)^p = \text{Hom}^p_{\mathcal X}(\mathcal D, \mathcal N)$, 
$\check C(\mathcal D, \mathcal N)^p = \text{\v CHom}^1_{\mathcal X}(\mathcal D, \mathcal N)$, etc.  
Hence we obtain our injection.

From a commutative diagram
\[
\begin{CD}
0 @>>> \mathcal N @>>> \mathcal I @>>> \mathcal L @>>> 0 \\
@. @VVV @VVV @VVV \\
0 @>>> \mathcal N' @>>> \mathcal I' @>>> \mathcal L' @>>> 0
\end{CD}
\]
we obtain the second result.
\end{proof}

Let $\mathcal N$ be a quasi-coherent $\mathcal R$-module.
We define $\mathcal N^{(i-1,i)} \in \text{Qcoh}(\mathcal R)$ by 
\[
\begin{split}
&\mathcal N^{(i-1,i)}_{i-1} = \mathcal N^{(i-1,i)}_{i-1,i} = \mathcal N^{(i-1,i)}_i = \mathcal N_{i-1,i}, \\
&\mathcal N^{(i-1,i)}_j = \mathcal N^{(i-1,i)}_{k-1,k} = \mathcal N^{(i-1,i)}_{\eta} = \mathcal N_{\eta}, 
\end{split}
\]
for $j \ne i-1,i$ and $k \ne i$.
Indeed, by the birationality condition, we confirm that $\mathcal N^{(i-1,i)} \in \text{Qcoh}(\mathcal R)$.
Let $\mathcal N^e = \bigoplus_{i=1}^n \mathcal N^{(i-1,i)}$.

Let $\mathcal N^t$ be the kernel of the natural homomorphism
$\mathcal N \to \mathcal N^e$.
Then we have $\mathcal N^t_{i-1,i} = 0$ for all $i$.
Hence $\text{\v CHom}^1(\mathcal M, \mathcal N^t) = 0$ for any $\mathcal M$.

\begin{Lem}
$\text{\v CHom}^1_{\mathcal X}(\mathcal D, \mathcal N^e) = 0$.
\end{Lem}

\begin{proof}
This is because $\text{Hom}_{\mathcal R_i}(\mathcal D_i, \mathcal N^e_i) \to 
\text{Hom}_{\mathcal R_{i-1,i}}(\mathcal D_{i-1,i}, \mathcal N^e_{i-1,i}))$ is surjective for all $i$.
\end{proof}

\begin{Lem}\label{inj2}
Let $\mathcal D$ be a divisorial sheaf and let
$\mathcal N$ be a quasi-coherent module.
Then there is an injective homomorphism 
\[
\text{\v CHom}^1_{\mathcal X}(\mathcal D, \mathcal N) \to \text{Hom}^1_{\mathcal X}(\mathcal D, \mathcal N).
\]
Moreover, for a homomorphism of quasi-coherent modules $\mathcal N \to \mathcal N'$, 
there is a commutative diagram
\[
\begin{CD}
\text{\v CHom}^1_{\mathcal X}(\mathcal D, \mathcal N) @>>> \text{Hom}^1_{\mathcal X}(\mathcal D, \mathcal N) \\
@VVV @VVV \\
\text{\v CHom}^1_{\mathcal X}(\mathcal D, \mathcal N') @>>> \text{Hom}^1_{\mathcal X}(\mathcal D, \mathcal N').
\end{CD}
\]
\end{Lem}

\begin{proof}
Let $\mathcal N^f = \mathcal N/\mathcal N^t$.
We have an exact sequence $0 \to \mathcal N^f \to \mathcal N^e \to \mathcal L \to 0$ for some $\mathcal L$.
Then we have a commutative diagram of exact sequences
\[
\begin{CD}
(\mathcal D, \mathcal N^e) @>>> (\mathcal D, \mathcal L) @>>>
\check C(\mathcal D, \mathcal N^f)^1 @>>> \check C(\mathcal D, \mathcal N^e)^1 \\
@V=VV @V=VV \\
(\mathcal D, \mathcal N^e) @>>> (\mathcal D, \mathcal L) @>>>
(\mathcal D, \mathcal N^f)^1 @>>> (\mathcal D, \mathcal N^e)^1.
\end{CD}
\]
Since $\check C(\mathcal D, \mathcal N^e)^1 = 0$, we obtain our injection 
$\text{\v CHom}^1(\mathcal D, \mathcal N^f) \to \text{Hom}^1(\mathcal D, \mathcal N^f)$.

From an exact sequence
\[
0 \to \mathcal N^t \to \mathcal N \to \mathcal N^f \to 0
\]
and $\check C(\mathcal D, \mathcal N^t)^1 = 0$, 
we deduce that $\check C(\mathcal D, \mathcal N)^1 \cong \check C(\mathcal D, \mathcal N^f)^1$.
Hence we have an injection $\check C(\mathcal D, \mathcal N)^1 \to (\mathcal D, \mathcal N)^1$.
The functoriality is obvious.
\end{proof}

\begin{Cor}
Let $\mathcal D$ be a divisorial sheaf and let
$\mathcal N$ be a quasi-coherent module.
Then there is a functorial isomorphism 
\[
\text{Hom}^1_{\mathcal X}(\mathcal D, \mathcal N) \to \text{\v CHom}^1_{\mathcal X}(\mathcal D, \mathcal N).
\]
\end{Cor}

\begin{proof}
By Lemma \ref{inj2}, we have $\text{\v CHom}^1(\mathcal D, \mathcal I) = 0$ for an injective module $\mathcal I$.
Then by the commutative diagram in the proof of Lemma \ref{inj1}, we obtain our isomorphism.
\end{proof}

\begin{Cor}
Let $\mathcal D$ be a divisorial sheaf and let $\mathcal N$ be quasi-coherent module.
Then $\text{Hom}^1_{\mathcal X}(\mathcal D, \mathcal N) \cong 0$ for $p \ne 0,1$.
\end{Cor}

\begin{proof}
We consider an exact sequence
\[
0 \to \mathcal N \to \mathcal I \to \mathcal N' \to 0
\]
where $\mathcal I$ is an injective object.
Since 
\[
\text{\v CHom}^1(\mathcal D, \mathcal I) \cong \text{Hom}^1(\mathcal D, \mathcal I) \cong 0
\]
we deduce that 
\[
\text{Hom}^1(\mathcal D, \mathcal N') \cong \text{\v CHom}^1(\mathcal D, \mathcal N') \cong 0.
\]
Then $\text{Hom}^2(\mathcal D, \mathcal N) \cong 0$.
Since $\mathcal N$ is general, we also have $\text{Hom}^2(\mathcal D, \mathcal N') \cong 0$, hence
$\text{Hom}^3(\mathcal D, \mathcal N) \cong 0$.
If we repeat this process, we obtain our result.
\end{proof}

\begin{Lem}
Let $\mathcal D = \mathcal R(\sum d_iD_i)$ and $\mathcal D' = \mathcal R(\sum d'_iD_i)$ be divisorial sheaves.
Assume that $d'_i \ge d_i$ for all $i$ except possibly $d'_{i_0} = d_{i_0} - 1$ for one $i_0$.
Then $\text{\v CHom}^1_{\mathcal X}(\mathcal D, \mathcal D') = 0$.
\end{Lem}

\begin{proof}
We have 
\[
\begin{split}
&\mathcal D_i \cong \mathcal D'_i \cong \mathcal R_i = k[t_0,\dots,t_n]\langle x_i,y_i \rangle/(x_iy_i - y_ix_i - t_0), \\
&\mathcal D_{i,i+1} \cong \mathcal D'_{i,i+1} \cong \mathcal R_{i,i+1} = k[t_0,\dots,t_n]\langle x_i,y_i,x_i^{-1} \rangle
/(x_iy_i - y_ix_i - t_0)
\end{split}
\]
with gluing
\[
\psi_{i,-}{\mathcal D}(r) = x_{i-1}^{d_i}r, \,\,\, \psi_{i,+}^{\mathcal D}(r) = r, \,\,\,
\psi_{i,-}^{\mathcal D'}(r) = x_{i-1}^{d'_i}r, \,\,\, \psi_{i,+}^{\mathcal D'}(r) = r
\]
and a transformation rule
\[
x_{i+1} = x_i^2y_i + t_{i+1}x_i, \,\,\, y_{i+1} = x_i^{-1}.
\]

Taking homomorphisms $h_{i-1,i}: \mathcal D_{i-1,i} \to \mathcal D'_{i-1,i}$ for $1 \le i \le n$, we look for 
homomorphisms $h_i: \mathcal D_i \to \mathcal D'_i$ for $0 \le i \le n$ which induce the $h_{i-1,i}$ in the sense that 
\[
h_{i-1,i} = h_{i-1} \otimes_{\mathcal R_{i-1}} \mathcal R_{i-1,i} - h_i \otimes_{\mathcal R_i} \mathcal R_{i-1,i} 
\]
for all $i$, that is 
\[
h_{i-1,i}(r) = h_{i-1}(r) - x_{i-1}^{d'_i}h_i(y_i^{d_i}r).
\]
We write
\[
h_i(r) = \sum_{l,m \ge 0} a^i_{l,m}x_i^ly_i^mr, \,\,\, h_{i-1,i}(r) = \sum_{l \ge 0, m \in \mathbf Z} b^i_{l,m}x_i^ly_i^mr
\]
for some $a^i_{l,m}, b^i_{l,m} \in k[t_0,\dots,t_n]$.
Then our condition becomes
\[
\sum_{l \ge 0, m \in \mathbf Z} b^i_{l,m}x_i^ly_i^m
= \sum_{l',m' \ge 0} a^{i-1}_{l',m'}x_{i-1}^{l'}y_{i-1}^{m'} - \sum_{l'',m'' \ge 0} a^i_{l'',m''}x_{i-1}^{d'_i}x_i^{l''}y_i^{m''+d_i}.
\]
We can write
\[
\begin{split}
&x_{i-1}^{l'}y_{i-1}^{m'} = x_i^{m'}y_i^{2m'-l'} + \text{ lower terms}, \\
&x_{i-1}^{d'_i}x_i^{l''}y_i^{m''+d_i} = x_i^{l''}y_i^{m''+d_i-d'_i} + \text{ lower terms},
\end{split}
\]
with respect to the bidegree defined by $\text{deg}(x_i,y_i) = (1,1)$ 
so that $x_i$ and $y_i$ become commutative modulo lower degree terms.
We note that the lower terms have smaller bidegrees by $(-e,-e)$ for some positive integers $e$. 

We will prove that all the $b^i_{l,m}$ are realized by some choices of the $a^{i-1}_{l',m'} $ and the $a^i_{l'',m''}$.
We proceed by the descending induction on the total degree $l+m$.

\vskip 1pc

We start with the case where $d'_i \ge d_i$ for all $i$.
We proceed be the descending induction on $i$.
First we consider the case $i = n$.
The terms $b^n_{l,m}$ with $l,m \ge 0$
are taken care of modulo lower terms by the $a^n_{l'',m''}$, while 
the remaining terms $b^n_{l,m}$ with $l \ge 0$ and $m < 0$ are by $a^{n-1}_{l',m'}$ 
with $l' > 2m' \ge 0$. 

We note that the lower degree terms are added to the given $b^n_{l,m}$ with lower degrees and treated by the 
$a^{n-1}_{l',m'} $ and the $a^n_{l'',m''}$ with lower degrees.
There are only finitely many non-zero $b^n_{l,m}$ initially, hence such $l,m$ are bounded.
Then $l',m',l'',m''$ are also bounded, even after taking care of lower degree terms.
Therefore we can obtain our claim by induction on the total degree.

Next we assume that the case $i = i_1$ is proved for some $i_1 \le n$ in the sense that 
the terms $b^{i_1}_{l,m}$
are taken care of modulo lower terms by the $a^{i_1}_{l'',m''}$ with $0 \le (n-i_1) l'' \le (n-i_1+1) m''$ and 
the $a^{i_1-1}_{l',m'}$ with $(n-i_1+1) l' > (n-i_1+2) m' \ge 0$. 
Thus we have still terms $a^{i_1-1}_{l',m'}$ with $0 \le (n-i_1+1) l' \le (n-i_1+2) m'$ to be used for the $b^{i_1-1}_{l,m}$. 

We will prove that all the $b^{i_1-1}_{l,m}$ are realized by some choices of the $a^{i_1-1}_{l'',m''}$ 
with $0 \le (n-i_1+1) l'' \le (n-i_1+2) m''$ and the $a^{i_1-2}_{l',m'}$ with $(n-i_1+2) l' > (n-i_1+3) m' \ge 0$. 
We proceed by the descending induction on the total degree $l+m$ again.
Indeed the $b^{i_1-1}_{l,m}$ with $0 \le (n-i_1+1) l \le (n-i_1+2) m$
are realized by the $a^{i_1-1}_{l'',m''}$ with $0 \le (n-i_1+1) l'' \le (n-i_1+2) m''$.
And the $b^{i_1-1}_{l,m}$ with $(n-i_1+1) l > (n-i_1+2) m$ and $l \ge 0$
are realized by the $a^{i_1-2}_{l',m'}$ with $(n-i_1+2) l' > (n-i_1+3) m' \ge 0$, 
because $(n - i_1 + 1) m' > (n - i_1 + 2)(2m' - l')$ is equivalent to $(n-i_1+2) l' > (n-i_1+3) m'$.
Therefore our assertion for the case where $d'_i \ge d_i$ for all $i$ is proved.

\vskip 1pc

Next we consider the case where $d'_i \ge d_i$ except $i = i_0$ and $d'_{i_0} = d_{i_0} - 1$.
The argument is similar, and we only check the difference from the previous argument.
We proceed by the descending induction on $i$.
We prove by the previous argument that the terms $b^i_{l,m}$ for $i > i_0$
are taken care of modulo lower terms by the $a^i_{l'',m''}$ with $0 \le (n-i) l'' \le (n-i+1) m''$ and 
the $a^{i-1}_{l',m'}$ with $(n-i+1) l' > (n-i+2) m' \ge 0$.

Then we consider the case $i = i_0$.
We can use the remaining terms $a^{i_0}_{l'',m''}$ with $0 \le (n-i_0) l'' \le (n-i_0+1) m''$ for the $b^{i_0}_{l,m}$. 
Then the terms $b^{i_0}_{l,m}$ with $0 \le (n-i_0) l \le (n-i_0+1) (m-1)$
are realized modulo lower terms by such $a^{i_0}_{l'',m''}$.
The remaining terms $b^{i_0}_{l,m}$ with $(n-i_0) l > (n-i_0+1) (m-1)$ and $l \ge 0$
are realized modulo lower terms by the $a^{i_0-1}_{l',m'}$ with
$(n-i_0 + 1) (l' +1) > (n-i_0+2) m' \ge 0$, because the last inequality is equivalent to 
$(n-i_0) m' > (n-i_0+1) (2m'-l'-1)$ and $m' \ge 0$.
Here we note that the forbidden case $l' = -1$ does not happen.
We note that, if we had $d'_{i_0} = d_{i_0} - 2$, then the case $l' = -1$ would appear. 

We have the remaining terms $a^{i_0-1}_{l'',m''}$ with 
$(n-i_0 + 1) l'' + n-i_0 + 1 \le (n-i_0+2) m''$ and $l'',m'' \ge 0$.
Then the $b^{i_0-1}_{l,m}$ with $(n-i_0 + 1) l + n-i_0 + 1 \le (n-i_0+2) m$ and $l \ge 0$ are 
taken care of modulo lower terms by such $a^i_{l'',m''}$.
We see that the remaining $b^{i_0-1}_{l,m}$ are those with
$(n-i_0 + 1) l + n-i_0 + 1 > (n-i_0+2) m$ and $l \ge 0$.
Then they are realized modulo lower terms by the $a^{i_0-1}_{l',m'}$ with $(n-i_0+2) l'  + (n - i_0 + 1) > (n-i_0+3) m' \ge 0$, $l' \ge 0$,
because the last inequality is equivalent to
$(n-i_0 + 1) m' + n-i_0 + 1 > (n-i_0+2) (2m' - l')$ and $l', m' \ge 0$. 
We note that $l' < 0$ does not happen, because $(n-i_0+2) l'  + (n - i_0 + 1) < 0$ otherwise.

For $i < i_0$, we claim that the terms 
$b^i_{l,m}$ are taken care of modulo lower terms by the $a^i_{l'',m''}$ with $(n-i) l'' + (n - i_0 + 1) \le (n-i+1) m''$, $l'' \ge 0$, and 
the $a^{i-1}_{l',m'}$ with $(n-i+1) l'  + (n - i_0 + 1) > (n-i+2) m' \ge 0$, $l' \ge 0$. 
We proceed by the descending induction on $i$.
The case $i = i_0 - 1$ is already proved.

Suppose that the case for $i = i_1$ is proved.
Then we can still use the $a^{i_1-1}_{l'',m''}$ with $(n-i_1+1) l''  + (n - i_0 + 1) \le (n-i_1+2) m''$, $l'' \ge 0$, $m'' \ge 0$ 
for the $b^{i_1-1}_{l,m}$
with $(n-i_1+1) l  + (n - i_0 + 1) \le (n-i_1+2) m$, $l \ge 0$.
We note that $m \ge 0$ is automatic.
The remaining $b^{i_1-1}_{l,m}$ with $(n-i_1+1) l  + (n - i_0 + 1) > (n-i_1+2) m$, $l \ge 0$
are taken care by the 
$a^{i_1-2}_{l',m'}$ with $(n-i_1+2) l'  + (n - i_0 + 1) > (n-i_1+3) m' \ge 0$, $l' \ge 0$, because the last inequality is equivalent to
$(n-i_1 + 1) m' + n-i_0 + 1 > (n-i_1+2) (2m' - l')$ and $m' \ge 0$. 
We note that $l' \ge 0$ is again automatic.
Therefore we complete the proof by the descending induction on $i$.
\end{proof}

The following says that $\mathcal T$ is a tilting locally free module:

\begin{Cor}
Let $\mathcal T = \mathcal R \oplus \bigoplus_{i=1}^n \mathcal R(-D_i)$.
Then $\text{Hom}^p(T, T) = 0$ for $p \ne 0$.
\end{Cor}

We need the following lemma later:

\begin{Lem}\label{SES}
There exists an exact sequence of $\mathcal R$-modules
\[
\begin{split}
&0 \to \mathcal R(\sum d_iD_i) \to \mathcal R(\sum d_iD_i + D_j) \oplus \mathcal R(\sum d_iD_i + D_k) \\
&\to \mathcal R(\sum d_iD_i + D_j + D_k) \to 0
\end{split}
\]
for any integers $j \le k$, 
where the first and the second arrows $(h,h')$ and $(g,g')$ are given by 
\[
\begin{split}
&h = (\dots, h_{j-1,j}, h_j, h_{j,j+1}, \dots) =  (\dots, x_{j-1}, 1, 1, \dots), \\
&h' = (\dots, h'_{k-1,k}, h'_k, h'_{k,k+1}, \dots) = (\dots, 1, y_k, y_k, \dots), \\
&g = (\dots, g_{k-1,j}, g_k, g_{k,k+1}, \dots) = (\dots, 1, y_k, y_k, \dots), \\
&g' = (\dots, g'_{j-1,j}, g'_j, g'_{j,j+1}, \dots) = (\dots, -x_{j-1}, -1, -1, \dots).
\end{split}
\]
\end{Lem}

\begin{proof}
First we check that these maps are compatible with the gluing maps: $h,h'$ are given by
\[
\begin{CD}
\mathcal R_{j-1,j} @<{x_{j-1}^{d_j}}<< \mathcal R_j @>1>> \mathcal R_{j,j+1} \\
@V{x_{j-1}}VV @V1VV @V1VV \\
\mathcal R_{j-1,j} @<{x_{j-1}^{d_j+1}}<< \mathcal R_j @>1>> \mathcal R_{j,j+1},
\end{CD}
\]
\[
\begin{CD}
\mathcal R_{k-1,k} @<{x_{k-1}^{d_k}}<< \mathcal R_k @>1>> \mathcal R_{k,k+1} \\
@V1VV @V{y_k}VV @V{y_k}VV \\
\mathcal R_{k-1,k} @<{x_{k-1}^{d_k+1}}<< \mathcal R_k @>1>> \mathcal R_{k,k+1}.
\end{CD}
\]
and $g,g'$ are similar.
We note that the dots have more complicated expressions (see Lemma \ref{global X}).

It is easy to see that the first arrow is injective, the second surjective, and the composite vanishes.
We will check the exactness at the middle term.
For example, assume that $j = k$ and $g_j(u) + g'_j(v) = 0$, i.e., $y_j u - v = 0$.
Then $v = h'_j(u)$.
Other cases are similar.
\end{proof}

\section{Comparison of $\text{End}(\mathcal T)$ and $\mathcal S$}

We will compare associative algebras $\mathcal S = k[s_0,\dots,s_n]\langle u,v \rangle \# G/(uv - vu - s)$ and 
$\mathcal A = \text{End}(\mathcal T)$, 
where $\mathcal T = \mathcal R \oplus \bigoplus_{i=1}^n \mathcal R(-D_i)$.
We will prove that there is an isomorphism of algebras $\mathcal S \to \text{End}(\mathcal T)$
over a linear coordinate change from the parameters $s_i$ to the $t_j$.

\subsection{Decomposition of $\mathcal S$ by orthogonal idempotents}

$\mathcal S$ has natural orthogonal idempotents 
corresponding to irreducible representations of $G$.
We define 
\[
e^s_i = \sum_{j=0}^n \zeta^{ij}g^j/(n+1) \in \mathcal S
\]
for $0 \le i \le n$, where $\zeta = e^{2\pi i/(n+1)}$ and $g$ is a generator of $\mathbf Z/(n+1)$.
We regard the index $i$ of $e^s_i$ is an element of $\mathbf Z/(n+1)$, i.e., we consider it only modulo $n+1$, 
we use the convention $e^s_{i+n+1} = e^s_i$.

The following is clear from the definition:

\begin{Lem}\label{unipotents}
(0) $e^s_0 = e$.

(1) $1 = \sum_{i=0}^n e^s_i$.

(2) $(e^s_i)^2 = e^s_i$, and $e^s_ie^s_j = 0$ for $i \ne j$.

(3) $e^s_iu = ue^s_{i+1}$ and $e^s_iv = ve^s_{i-1}$.
\end{Lem}

Thus we have a decomposition as $k[s_0,\dots,s_n]$-module:
\[
\mathcal S = \bigoplus_{i,j=0}^n \mathcal S_{i,j}, \,\,\, \mathcal S_{i,j} = e^s_i \mathcal S e^s_j.
\]
We define elements $\alpha^s_{i,i+1} \in \mathcal S_{i,i+1}$ and 
$\beta^s_{i+1,i} \in \mathcal S_{i+1,i}$ by
\[
\alpha^s_{i,i+1} = e^s_iue^s_{i+1}, \,\,\, \beta^s_{i+1,i} = e^s_{i+1}ve^s_i
\]
for $0 \le i \le n$.

\begin{Lem}
(1) $u = \sum_{i=0}^n \alpha^s_{i,i+1}$ and $v = \sum_{i=0}^n \beta^s_{i+1,i}$.

(2) $\mathcal S$ is generated as a $k[s_0,\dots,s_n]$-algebra by the $e^s_i$, $\alpha^s_{i,i+1}$ and 
$\beta^s_{i+1,i}$ for $0 \le i \le n$.

(2) There are following relations:
\[
\begin{split}
&\alpha^s_{i,i+1}\beta^s_{i+1,i} = uv e^s_i, \\
&\beta^s_{i+1,i}\alpha^s_{i,i+1} = vu e^s_{i+1}, \\
&\alpha^s_{i,i+1} \alpha^s_{i+1,i+2} \dots \alpha^s_{i+n-1,i+n} \alpha^s_{i+n,i+n+1} = u^{n+1} e^s_i, \\
&\beta^s_{i,i-1} \beta^s_{i-1,i-2} \dots \beta^s_{i-n+1,i-n} \beta^s_{i-n,i-n-1} = v^{n+1} e^s_i.
\end{split}
\]
\end{Lem}

\begin{proof}
(1) If $j \ne i+1$, then
\[
e^s_iu e^s_j = e^s_ie^s_iu e^s_j = e^s_iue^s_{i+1} e^s_j = 0.
\]
Hence 
\[
u = \sum_{i,j=0}^n e^s_i u e^s_j = \sum_{i=0}^n e^s_i u e^s_{i+1}.
\]
The assertion for $v$ is similarly proved.

(2) The claim is clear.

(3) The assertion follows from $uv e^s_i = e^s_i uv$, $u^{n+1}e^s_i = e^s_i u^{n+1}$ and $v^{n+1}e^s_i = e^s_i v^{n+1}$.
\end{proof}

\subsection{Decomposition of $\mathcal A$ by orthogonal idempotents}

We will determine the structure of the endomorphism algebra 
$\mathcal A = \text{End}_{\mathcal R}(\mathcal T)$ of the tilting object 
$\mathcal T = \mathcal R \oplus \bigoplus_{i=1}^n \mathcal R(-D_i)$.

$\mathcal A$ has natural orthogonal idempotents 
$e^t_i \in \mathcal A$ for $0 \le i \le n$ defined by the projections
\[
e^t_i: \mathcal T \to \mathcal R(-D_i) \to \mathcal T
\]
where we set $D_0 = 0$ and $\mathcal R(-D_0) = \mathcal R$.
We have a decomposition of the identity to orthogonal idempotents $\sum_{i=0}^n e^t_i = 1 \in \mathcal A$.
Thus we have a decomposition as $k[t_0,\dots,t_n]$-modules:
\[
\mathcal A = \bigoplus_{i,j=0}^n \mathcal A_{i,j}, \,\,\, \mathcal A_{i,j} = e^t_i \mathcal A e^t_j.
\]
We note that we use again the convention $e^t_i = e^t_{i+n+1}$.

We will express any element $h \in \mathcal A_{i,j}$ by a sequence of functions 
$(h_0,h_1,\dots,h_n)$, where $h_k \in \mathcal R_k$, such that
$h_k(r) = h_kr$, because other components $h_{k-1,k}$ are determined by the $h_k$.

\begin{Lem}\label{alphabeta}
There are elements $\alpha^t_{i,i+1} \in \mathcal A_{i,i+1}$, $\beta^t_{i+1,i} \in \mathcal A_{i+1,i}$ 
for $0 \le i \le n$ such that
\[
(\alpha^t_{i,i+1})_k = \begin{cases} 
x_ky_k + (i-k-1)t_0 + t_{k+1} + \dots + t_i, \,\,\, &(0 \le k < i), \\
x_i, \,\,\, &(k=i), \\
1, \,\,\, &(i < k \le n). \end{cases}
\]
\[
(\beta^t_{i+1,i})_k = \begin{cases} 
1, \,\,\, &(0 \le k < i), \\
y_i, \,\,\, &(k = i), \\
x_ky_k - (k - i)t_0 - t_{i+1} - \dots - t_k, \,\,\, &(i < k \le n). \end{cases}
\]
\end{Lem}

\begin{proof}
Since $x_{i+1} = x_i^2y_i + t_{i+1}x_i$ and $y_{i+1} = x_i^{-1}$, we have
\[
x_{i+1}y_{i+1} = x_iy_i + t_0 + t_{i+1}.
\]
The following is a part of the diagram for the homomorphism 
$\alpha^t_{i,i+1}: \mathcal R(-D_{i+1}) \to \mathcal R(-D_i)$: 
\[
\begin{CD}
\mathcal R_{i-1} @>1>> \mathcal R_{i-1,i} @<1<< \mathcal R_i @>1>> \mathcal R_{i,i+1} @<{x_i^{-1}}<< R_{i+1} \\
@V{x_{i-1}y_{i-1}+t_i}VV @V{x_{i-1}y_{i-1}+t_i}VV @V{x_i}VV @V{x_i}VV @V1VV  \\
\mathcal R_{i-1} @>1>> \mathcal R_{i-1,i} @<{x_{i-1}^{-1}}<< \mathcal R_i @>1>> \mathcal R_{i,i+1} @<1<< R_{i+1}.
\end{CD}
\]
We confirm the commutativity of the middle left square by $x_{i-1}^{-1}x_i = y_ix_i = x_{i-1}y_{i-1} + t_i$.
We note that the case $i = 0, n$ are also confirmed by the following diagrams for $\alpha^t_{0,1}$ and $\alpha^t_{n,0}$:
\[
\begin{CD}
\mathcal R_0 @>1>> \mathcal R_{0,1} @<{x_0^{-1}}<< \mathcal R_1 \\
@V{x_0}VV @V{x_0}VV @V1VV \\
\mathcal R_0 @>1>> \mathcal R_{0,1} @<1<< \mathcal R_1, 
\end{CD}
\]
\[
\begin{CD}
\mathcal R_{n-1} @>1>> \mathcal R_{n-1,n} @<1<< \mathcal R_n \\
@V{x_{n-1}y_{n-1}+t_n}VV @V{x_{n-1}y_{n-1}+t_n}VV @V{x_n}VV \\
\mathcal R_{n-1} @>1>> \mathcal R_{n-1,n} @<{x_{n-1}^{-1}}<< \mathcal R_n.
\end{CD}
\]
Then the assertion for $k < i-1$ is confirmed by the descending induction on $k$ by
\[
x_{k+1}y_{k+1} + (i-k-2)t_0 + t_{k+2} + \dots + t_i
= x_ky_k + (i-k-1)t_0 + t_{k+1} + \dots + t_i.
\]

The assertion for $\beta^t_{i+1,i}$ with $\beta^t_{1,0}$ and $\beta^t_{0,n}$ is similarly confirmed by the following diagrams:
\[
\begin{CD}
\mathcal R_{i-1} @>1>> \mathcal R_{i-1,i} @<{x_{i-1}^{-1}}<< \mathcal R_i @>1>> \mathcal R_{i,i+1} @<1<< R_{i+1} \\
@V1VV @V1VV @V{y_i}VV @V{y_i}VV @VV{x_{i+1}y_{i+1} - t_0 - t_{i+1}}V  \\
\mathcal R_{i-1} @>1>> \mathcal R_{i-1,i} @<1<< \mathcal R_i @>1>> \mathcal R_{i,i+1} @<{x_i^{-1}}<< R_{i+1}.
\end{CD}
\]
\[
\begin{CD}
\mathcal R_0 @>1>> \mathcal R_{0,1} @<1<< \mathcal R_1 \\
@V{y_0}VV @V{y_0}VV @VV{x_0y_0 - t_0 - t_1}V \\
\mathcal R_0 @>1>> \mathcal R_{0,1} @<{x_0^{-1}}<< \mathcal R_1.
\end{CD}
\]
\[
\begin{CD}
\mathcal R_{n-1} @>1>> \mathcal R_{n-1,n} @<{x_{n-1}^{-1}}<< \mathcal R_n \\
@V1VV @V1VV @V{y_n}VV \\
\mathcal R_{n-1} @>1>> \mathcal R_{n-1,n} @<1<< \mathcal R_n
\end{CD}
\]
and
\[
x_{k+1}y_{k+1} - (k + 1 - i)t_0 - t_{i+1} - \dots - t_{k+1}
= x_ky_k - (k - i)t_0 - t_{i+1} - \dots - t_k.
\]
\end{proof}

\begin{Defn}
We define $u^t, v^t \in \mathcal A$ by
\[
u^t = \sum_{i=0}^n \alpha^t_{i,i+1}, \,\,\, v^t = \sum_{i=0}^n \beta^t_{i+1,i}.
\]
\end{Defn}

Then we have $\alpha^t_{i,i+1} = e^t_i u e^t_{i+1}$, $\beta^t_{i+1,i} = e^t_{i+1} v e^t_i$.

We can express $u^t$ and $v^t$ as matrix forms when they are restricted to an affine chart $\mathcal R_0$:

\begin{Lem}
\[
\begin{split}
&u^t \vert_{\mathcal R_0} =  \\
&\left( \begin{matrix}
0 & x_0 & 0 & \dots & 0 \\
0 & 0 & x_0y_0 + t_1 & \dots & 0 \\
\dots & \dots & \dots & \dots & \dots \\ 
0 & 0 & 0 & \dots & x_0y_0 + (n-2)t_0 + t_1 + \dots + t_{n-1} \\
x_0y_0 + (n-1)t_0 + t_1 + \dots + t_n & 0 & 0 & \dots & 0
\end{matrix} \right)
\end{split}
\]
\[
v^t \vert_{\mathcal R_0} = \left( \begin{matrix}
0 & 0 & \dots & 0 & 1 \\
y_0 & 0 & \dots & 0 & 0 \\
0 & 1 & \dots & 0 & 0 \\ 
\dots & \dots & \dots & \dots & \dots \\ 
0 & 0 & \dots & 1 & 0
\end{matrix} \right)
\]
\end{Lem}

We have the following identities:

\begin{Cor}\label{uv}
\[
\begin{split}
u^tv^t \vert_{\mathcal R_0} &= \text{diag}(x_0y_0, x_0y_0 + t_1, x_0y_0 + t_0 + t_1 + t_2, \\
&\dots, x_0y_0 + (n-1)t_0 + t_1 + \dots + t_n), \\
v^tu^t \vert_{\mathcal R_0} &= \text{diag}(x_0y_0 + (n-1)t_0 + t_1 + \dots + t_n, x_0y_0 - t_0, x_0y_0 + t_1, \\
&\dots, x_0y_0 + (n-2)t_0 + t_1 + \dots + t_{n-1}), \\
(u^tv^t - v^tu^t) \vert_{\mathcal R_0} &= \text{diag}(- (n-1)t_0 - t_1 - \dots - t_n, t_0 + t_1, t_0 + t_2, \dots, t_0 + t_n).
\end{split}
\]
\end{Cor}

\subsection{Algebra homomorphism $\mathcal S \to \mathcal A$}

We defined $e_i^s = \frac 1{n+1} \sum_{j = 0}^n \zeta^{ij}g^j$.
Thus we have $g^j = \sum_{i=0}^n \zeta^{-ij}e^s_i$.
Now we define
\[
g^t = \sum_{i = 0}^n \zeta^{-i}e^t_i.
\]



\begin{Lem}\label{gu}
(1) $(g^t)^{n+1} = 1$ and $(g^t)^j \ne 1$ for $0 < j < n+1$.

(2) $e^t_i = \sum_{j=0}^n \zeta^{ij}(g^t)^j/(n+1)$.

(3) $g^tu^t = \zeta u^tg^t$ and $g^tv^t = \zeta^{-1} v^tg^t$.
\end{Lem}

\begin{proof}
(1) This is because $(g^t)^j = \sum_{i=0}^n \zeta^{-ij}e^t_i$.

(2) We have $\sum_{j=0}^n \zeta^{ij}(g^t)^j = \sum_{j,k=0}^n \zeta^{ij}\zeta^{-jk}e^t_k = (n+1)e^t_i$.

(3) We will prove that $g^tu^te^t_{i+1} = \zeta u^tg^te^t_{i+1}$ and $g^tv^te^t_i = \zeta^{-1} v^tg^te^t_i$
for all $i$.
Indeed
\[
\begin{split}
&g^tu^te^t_{i+1} = \zeta^{-i}e^t_i\alpha^t_{i,i+1}e^t_{i+1}
= \zeta \alpha_{i,i+1} \zeta^{-i-1}e^t_{i+1} = \zeta u^tg^t e^t_{i+1}, \\
&g^tv^te^t_i = \zeta^{-i-1}e^t_{i+1}\beta^t_{i+1,i}e^t_i 
= \zeta^{-1}\beta^t_{i+1,i}\zeta^{-i}e^t_i = \zeta^{-1}v^tg^te^t_i.
\end{split}
\]
\end{proof}

We define $w_0, \dots, w_n$ by
\[
s_i = \frac 1{n+1}\sum_{j=0}^n \zeta^{ij}w_j.
\]
We note that this linear transformation is invertible.
Then we compare two formulas:
\begin{itemize}
\item $\mathcal S$: $uv - vu = \sum_{i=0}^n s_ig^i =\sum_{j=0}^n w_je^s_j$.

\item $\mathcal X$: 
$u^tv^t - v^tu^t = (- (n-1)t_0 - t_1 - \dots - t_n)e^t_0 + (t_0 + t_1)e^t_1 \break + (t_0 + t_2)e^t_2 + \dots + (t_0 + t_n)e^t_n$.
\end{itemize}
From the comparison, we define another invertible linear transformation:
\[
w_0 = - (n-1)t_0 - t_1 - \dots - t_n, \,\,\, w_1 = t_0 + t_1, \dots, w_n = t_0 + t_n.
\]

We define an algebra homomorphism 
\[
\phi: \mathcal S \to \mathcal A
\]
over the linear coordinate change $k[s_0,\dots,s_n] \to k[t_0,\dots,t_n]$ defined above by  
\[
\phi(e^s_i) = e^t_i, \,\,\, \phi(u) = u^t, \,\,\, \phi(v) = v^t, \,\,\, \phi(g) = g^t.
\]
By Corollary \ref{uv}, $\phi$ is a well defined algebra homomorphism.

\begin{Thm}
$\phi:  \mathcal S \to \mathcal A$ is an algebra isomorphism. 
\end{Thm}

We will prove that $\phi$ is injective and surjective in the following subsections.

\subsection{Injectivity of $\phi$}

We call the subalgebra $e \mathcal S e$ of $\mathcal S$ the {\em ring of $G$-invariants}, where 
$e = e^s_0$ (\cite{CBH}). 
We denote the subalgebra $e^t_0 \mathcal A e^t_0 = \text{Hom}(\mathcal R, \mathcal R)$ by
$\Gamma(\mathcal R)$, and call it the {\em ring of global functions}.

\begin{Lem}
The injectivity of $\phi$ is reduced to the injectivity of its restriction 
$\phi_e: e \mathcal S e \to \Gamma(\mathcal R)$.
\end{Lem}

\begin{proof}
Assume that $\text{Ker}(\phi) \ne 0$.
Since $\phi$ preserves idempotents, we have $\phi(e^s_i\mathcal S e^s_j) \subset e^t_i\mathcal A e^t_j$.
Thus there are $i,j$ such that $\text{Ker}(\phi) \cap e^s_i\mathcal S e^s_j \ne 0$.

If we set $\text{deg}(s_i,g,u,v) = (0,0,1,1)$ and introduce a filtration 
$F_p(\mathcal S) = \{x \in \mathcal S \mid \text{deg}(x) \le p\}$, 
then $\text{Gr}(\mathcal S) \cong k[s_0,\dots,s_n][u,v] \# G$.
Hence $\mathcal S$ has no non-zero divisor.
By multiplying $u^i$ and $v^j$ from the left and the right respectively, we obtain 
$\text{Ker}(\phi) \cap e \mathcal S e \ne 0$, because we have
$u^i e^s_i\mathcal S e^s_j v^j \subset e \mathcal S e$.
\end{proof}

We define $x^t_{i,i},y^t_{i,i},z^t_{i,i} \in \mathcal A_{i,i}$ for $0 \le i \le n$ by
\[
x^t_{i,i} = e^t_i(u^t)^{n+1} e^t_i, \,\,\, y^t_{i,i}= e^t_i(v^t)^{n+1} e^t_i, \,\,\, z^t_{i,i} = e^t_i u^tv^t e^t_i.
\]
In particular we denote $x^t = x^t_{0,0}, y^t = y^t_{0,0}, z^t = z^t_{0,0}$. 
We have $x^t,y^t,z^t \in \Gamma(\mathcal R)$.

\begin{Lem}\label{alphabeta2}
(1) 
\[
\begin{split}
&(\alpha^t_{i,i+1} \alpha^t_{i+1,i+2} \dots \alpha^t_{i+n-1,i+n} \alpha^t_{i+n,i+n+1})_n = x_n, \\
&(\beta^t_{i,i-1} \beta^t_{i-1,i-2} \dots \beta^t_{i-n+1,i-n} \beta^t_{i-n,i-n-1})_0 = y_0. \\
\end{split}
\]

(2) 
\[
\begin{split}
&(\alpha^t_{i,i+1}\beta^t_{i+1,i})_0 = \begin{cases} 
x_0y_0, \,\,\, &(i = 0), \\
x_0y_0 + (i-1)t_0 + t_1 + \dots + t_i, \,\,\, &(0 < i \le n), \end{cases} \\
&(\alpha^t_{i,i+1}\beta^t_{i+1,i})_n = \begin{cases} 
x_ny_n - (n-i)t_0 - t_{i+1} - \dots - t_n, \,\,\, &(0 \le i < n), \\
x_ny_n, \,\,\, &(i = n). \end{cases}
\end{split}
\]
\end{Lem}

\begin{proof}
The assertion follows from
\[
\begin{split}
&(\alpha^t_{i,i+1})_0 = \begin{cases} x_0, \,\,\, &(i = 0),  
\\ x_0y_0 + (i-1)t_0 + t_1 + \dots + t_i, \,\,\, &(i \ne 0), \end{cases}, \\
&(\alpha^t_{i,i+1})_n = \begin{cases} 1, \,\,\, &(i \ne n),  \\ x_n, \,\,\, &(i = n), \end{cases} \\
&(\beta^t_{i+1,i})_0 = \begin{cases} y_0, \,\,\, &(i = 0), \\ 1, \,\,\, &(i \ne 0), \end{cases}, \\
&(\beta^t_{i+1,i})_n = \begin{cases} x_ny_n - (n-i)t_0 - t_{i+1} - \dots - t_n, \,\,\, &(i \ne n), 
\\ y_n, \,\,\, &(i = n). \end{cases}
\end{split}
\]  
\end{proof}

\begin{Cor}
\[
\begin{split}
&(x^t)_0 = x_0 \prod_{i=1}^n (x_0y_0 + (i-1)t_0 + t_1 \dots + t_i), \,\,\, (y^t)_0 = y_0, \,\,\, (z^t)_0 = x_0y_0, \\
&(x^t)_n = x_n, \,\,\, (y^t)_n = y_n \prod_{i=1}^n (x_ny_n - (n + 1- i)t_0 - t_i \dots - t_n), \\
&(z^t)_n = x_ny_n - nt_0 - t_1 - \dots - t_n.
\end{split}
\]
\end{Cor}

We note that the products $\prod_{i=1}^n$ in the above formula do not depend on the order of the factors.

\begin{proof}
Indeed we have
\[
\begin{split}
&e^t_0(u^t)^{n+1} e^t_0 = \alpha^t_{0,1} \alpha^t_{1,2} \dots \alpha^t_{n-1,n} \alpha^t_{n,0}, \\
&e^t_0(v^t)^{n+1} e^t_0 = \beta^t_{0,n} \beta^t_{n,n-1} \dots \beta^t_{2,1} \beta^t_{1,0}, \\
&e^t_0 u^tv^t e^t_0 = \alpha^t_{0,1} \beta^t_{1,0}.
\end{split}
\] 
\end{proof}

\begin{proof}[Proof of the injectivity]
We define 
\[
x^s = e u^{n+1}e, \,\,\, y^s = ev^{n+1}e, \,\,\, z^s = euve.
\]
Then we have 
\[
\phi(x^s) = x^t, \,\,\, \phi(y^s) = y^t, \,\,\, \phi(z^s) = z^t. 
\]

We recall that, if we define a filtration on $\mathcal S$ by setting degrees $\text{deg}(s_i,g,u,v) = (0,0,1,1)$, then 
$\text{Gr}(\mathcal S) \cong k[s_0,\dots,s_n][u,v] \# G$.
We define a filtration on $\mathcal R_0$ by setting degrees 
\[
\text{deg}(t_i,x_0,y_0) = (0,-n+1,n+1).
\]
Then we have $\text{Gr}(\mathcal R_0) \cong k[t_0,\dots,t_n][x_0,y_0]$.
Moreover $\phi_e: e \mathcal S e \to \Gamma(\mathcal R)$ preserves the degrees, so that  
$\phi_e$ induces a homomorphism 
$\text{Gr}(\phi_e): \text{Gr}(e\mathcal S e) \to \text{Gr}(\mathcal R_0)$.

We note that $e\mathcal S e$ is generated by $x^s$, $y^s$ and $z^s$ over $k[s_0,\dots,s_n]$.
The only relation among generators $x^s$, $y^s$, $z^s$ in $\text{Gr}(\mathcal S) \cong k[s_0,\dots,s_n][u,v]^G$ 
is $x^sy^s \equiv (z^s)^{n+1}$.
We also note that the only relation among $(x^t)_0$, $(y^t)_0$, $(z^t)_0$ in $\text{Gr}(\mathcal R_0)$ 
is $(x^t)_0(y^t)_0 \equiv (z^t)_0^{n+1}$.
Therefore $\text{Gr}(\phi_e)$ is injective, hence so is $\phi_e$.
\end{proof}

\subsection{Surjectivity of $\phi$}

We will prove the surjectivity.
First we consider a $k[t_0,\dots,t_n]$-algebra $\mathcal A_{i,i} = \text{End}(\mathcal R(-D_i))$ for $0 \le i \le n$.
We define $x^t_{i,i},y^t_{i,i},z^t_{i,i} \in \mathcal A_{i,i}$ by
\[
x^t_{i,i} = e^t_i(u^t)^{n+1} e^t_i, \,\,\, y^t_{i,i}= e^t_i(v^t)^{n+1} e^t_i, \,\,\, z^t_{i,i} = e^t_i u^tv^t e^t_i.
\]
These elements are contained in the image of $\phi$.
We claim that $\mathcal A_{i,i}$ is generated by $x^t_{i,i}, y^t_{i,i}, z^t_{i,i}$
as an algebra over $k[t_0,\dots,t_n]$.

\begin{Lem}
\[
\begin{split}
&(y^t_{i,i})_0 = y_0, \,\,\, (z^t_{i,i})_0 = \begin{cases} x_0y_0, \,\,\, &(i = 0), \\
x_0y_0 + (i-1)t_0 + t_1 + \dots + t_i, \,\,\, &(0 < i \le n), \end{cases} \\
&(x^t_{i,i})_n = x_n, \,\,\, (z^t_{i,i})_n = \begin{cases} 
x_ny_n - (n-i)t_0 - t_{i+1} - \dots - t_n, \,\,\, &(0 \le i < n), \\
x_ny_n, \,\,\, &(i = n). \end{cases}
\end{split}
\]
\end{Lem}

\begin{proof}
Indeed we have
\[
\begin{split}
&e^t_i(u^t)^{n+1} e^t_i = \alpha^t_{i,i+1} \alpha^t_{i+1,i+2} \dots \alpha^t_{i+n-1,i+n} \alpha^t_{i+n,i}, \\
&e^t_i(v^t)^{n+1} e^t_i = \beta^t_{i,i-1} \beta^t_{i-1,i-2} \dots \beta^t_{i+2,i+1} \beta^t_{i+1,i}, \\
&e^t_i u^tv^t e^t_i = \alpha^t_{i,i+1} \beta^t_{i+1,i}.
\end{split}
\]
\end{proof}

\begin{proof}[Proof of the surjectivity]
We first prove the surjectivity of $\phi$ for $\mathcal A_{i,i}$. 
We take an arbitrary $h \in \mathcal A_{i,i} = \text{Hom}(\mathcal R(-D_i),\mathcal R(-D_i))$, 
and let $r_k = h_k(1)$.  
We define degree by $\text{deg}(t_j, x_k,y_k) = (0,1,-1)$.
Then all relations in $\mathcal R_k$ as well as gluings for $\mathcal R(-D_i)$ are homogeneous.
Hence we may assume that the $r_k$ are homogeneous 
with a fixed degree $\text{deg}(r)$.

We claim that, if $\text{deg}(r) \le 0$, then we can express $h$ by using $y_t^{(i)}, z_t^{(i)} \in \mathcal A_{i,i}$
with coefficients in $k[t_0, \dots, t_n]$. 
For this purpose, we look at the affine chart $\mathcal R_0$.
We consider the secondary grading in $\mathcal R_0$ defined by $\text{deg'}(t_j, x_0,y_0) = (0,1,1)$.
Then we can write
\[
r_0 = a_0(x_0y_0)^by_0^c + \text{ lower deg' terms}
\]
for $a_0 \in k[t_0,\dots,t_n]$ and $b,c \ge 0$ using the commutation relation $x_0y_0 = y_0x_0 + t_0$, 
because $\text{deg}(r) \le 0$.
Let $h' = h - a_0(z^t_{i,i})^b(y^t_{i,i})^c \in \mathcal A_{i,i}$.
Then $h'$ is again homogeneous of degree $-c$, and the secondary degree is less than $2b + c$.
Therefore we infer that $h$ is expressed by $y^t_{i,i},z^t_{i,i}$ using induction on the secondary degree.

We assume next that $\text{deg}(r) > 0$.
We consider the secondary grading in $\mathcal R_n$ defined by $\text{deg''}(t_j,x_n,y_n) = (0,1,1)$.
Then we can write
\[
r_n = a_nx_n^b(x_ny_n)^c + \text{ lower deg'' terms}
\]
for $a_n \in k[t_0,\dots,t_n]$ and $b,c \ge 0$ using the commutation relation $x_ny_n = y_nx_n + t_0$, 
because $\text{deg}(r) > 0$.
Let $h'' = h - a_n(x^t_{i,i})^b(z^t_{i,i})^c \in \mathcal A_{i,i}$.
Then $h''$ is again homogeneous of degree $b$, and secondary degree less than $b + 2c$.
Therefore we conclude that $h$ is expressed by $x^t_{i,i},z^t_{i,i}$ using induction on the secondary degree.

\vskip 1pc

Now we consider $\mathcal A_{i,j}$ for $i \ne j$.
We take $h \in \mathcal A_{i,j}$ and let $r_k = h_k(1)$ as before.
We may assume that the $r_k$ are homogeneous, but the gluing does not preserve the degree.
Indeed we have 
\[
\text{deg}(r_n) = \text{deg}(r_0) + \begin{cases} 0, \,\,\, &(i,j \ne 0), \\
-1, \,\,\, &(i = 0), \\
1, \,\,\, &(j = 0).  \end{cases} 
\] 

Assume first that (i) or (ii) holds:
\begin{enumerate}
\item[(i)] $\text{deg}(r_0) \le 0$. 
$0 \ne j < i$ or $i = 0$.

\item[(ii)] $\text{deg}(r_0) < 0$. 
$0 \ne i < j$ or $j = 0$.
\end{enumerate}

We have 
\[
(\beta^t_{i,i-1} \dots \beta^t_{j+1,j})_0 = \begin{cases} 1, \,\,\, &(0 < j < i \text{ or } i = 0), \\
y_0, \,\,\, &(0 < i < j \text{ or } j = 0). \end{cases}
\]
We can write
\[
r_0 = a_0(x_0y_0)^by_0^c + \text{ lower deg' terms}
\]
with $a_0 \in k[t_0,\dots,t_n]$, $b \ge 0$, and (i) $c \ge 0$ or (ii) $c > 0$.
We set 
\[
h' = h - a_0(z^t_{i,i})^b(y^t_{i,i})^{c'}\beta^t_{i,i-1} \dots \beta^t_{j+1,j} \in \mathcal A_{i,j}
\]
where (i) $c' = c$, or (ii) $c' = c - 1$.
Then $h'_0$ is again homogeneous of degree $-c$, and the secondary degree is less than $2b + c$.
Therefore we infer that $h$ is expressed by $y^t_{i,i},z^t_{i,i}$ and the $\beta_{k+1,k}$
using induction on the secondary degree.

Assume next that (i') or (ii') holds:
\begin{enumerate}
\item[(i')] $\text{deg}(r_n) \ge 0$. 
$i < j$ or $j = 0$.

\item[(ii')] $\text{deg}(r_n) > 0$. 
$0 \ne j < i$.
\end{enumerate}

We note that the cases (i), (ii), (i'), (ii') exhaust all cases.
Indeed, in the case $i,j \ne 0$, if $j < i$, then (i) or (ii') holds, and if $i < j$, then (i') or (ii) holds. 
In the case $i = 0$, if (i) does not hold, then then (i') holds.
In the case $j = 0$, if (ii) does not hold, then (i') holds.

We have 
\[
(\alpha^t_{i,i+1} \dots \alpha^t_{j-1,j})_n = \begin{cases} 1, \,\,\, &(i < j), \\
x_n, \,\,\, &(j < i). \end{cases}
\]
We can write
\[
r_n = a_n(x_ny_n)^bx_n^c + \text{ lower deg'' terms}
\]
with $a_n \in k[t_0,\dots,t_n]$, $b \ge 0$, and (i') $c \ge 0$ or (ii') $c > 0$.
We set 
\[
h' = h - a_n(z^t_{i,i})^b(x^t_{i,i})^{c'}\alpha^t_{i,i+1} \dots \alpha^t_{j-1,j} \in \mathcal A_{i,j}
\]
where (i) $c' = c$, or (ii) $c' = c - 1$.
Then $h'_n$ is again homogeneous of degree $c$, and the secondary degree is less than $2b + c$.
Therefore we infer that $h$ is expressed by $x^t_{i,i},z^t_{i,i}$ and the $\alpha_{k,k+1}$
using induction on the secondary degree.
\end{proof}

\begin{Rem}
We have $s_0 = \frac 1{n+1} \sum_{j=0}^n w_j = t_0$.
The commutative deformations of $S$ and $X$ are defined by $s_0 = 0$ and $t_0 = 0$ respectively.
Thus the parameters $s_0$ and $t_0$ correspond to the direction of purely non-commutative deformations.
\end{Rem}

\section{Derived equivalence}

Let $\mathcal T$ be a coherent locally free $\mathcal R$-module.
It is called a {\em tilting generator} of $D^b(\text{coh}(\mathcal R))$ if the following conditions are 
satisfied:

(1) (tilting) $\text{Hom}(\mathcal T, \mathcal T[m]) = 0$ for $m \ne 0$.

(2) (generator) For any object $x \in D^-(\text{coh}(\mathcal R))$, if $\text{Hom}(\mathcal T, x[m]) = 0$ for
all integers $m$, then $x \cong 0$.

\begin{Thm}
$\mathcal T = \mathcal R \oplus \bigoplus_{i=1}^n \mathcal R(-D_i)$ is a tilting generator
of $D^-(\text{coh}(\mathcal R))$.
\end{Thm}

\begin{proof}
We know already that $\mathcal T$ is a tilting object.
We will prove that, for any non-zero object $x \in D^-(\text{coh}(\mathcal R))$, 
there exists an integer $m$ such that $\text{Hom}(\mathcal T, x[m]) \ne 0$. 
$x$ is quasi-isomorphic to the following type of complex
\[
\begin{CD}
\dots @>>> \mathcal M^{p-1} @>h>> \mathcal M^p @>>> 0
\end{CD}
\]
such that $h$ is not surjective. 
By Lemmas \ref{generator} and \ref{generator1}, there exists a divisorial sheaf 
$\mathcal D = \mathcal R(\sum d_iD_i)$ such that $\text{Hom}(\mathcal D, \text{Coker}(h)) \ne 0$
and $\text{Hom}(\mathcal D, \text{Im}(h)[1]) = 0$.
Then it follows that there is a homomorphism $\mathcal D \to \mathcal M^p$ which does not factor 
through $h$.
Therefore we have $\text{Hom}(\mathcal D, x[p]) \ne 0$.

On the other hand, assume that $\text{Hom}(\mathcal T, x[m]) = 0$ for all $m$.
Then by Lemma \ref{SES}, we have $\text{Hom}(\mathcal D, x[m]) = 0$ for all divisorial sheaves $\mathcal D$, 
a contradiction.
Therefore we have our claim.
\end{proof}

Let $\mathcal A = \text{End}(\mathcal T)$ be the endomorphism algebra.
Since it is isomorphic to $\mathcal S$, it is a Noetherian algebra.

\begin{Lem}\label{coherent}
Let $M$ be a coherent module on $\mathcal X$.
Then $\text{Hom}(\mathcal T, \mathcal M[p])$ is a 
finitely generated right $\mathcal A$-modules for $p = 0,1$, and 
$\text{Hom}(\mathcal T, \mathcal M[p]) = 0$ for $p \ne 0,1$.
\end{Lem}

\begin{proof}
We know already that $\text{Hom}(\mathcal T, \mathcal M[p]) = 0$ for $p \ne 0,1$.

$\text{Hom}(\mathcal T, \mathcal R)$ and $\text{Hom}(\mathcal T, \mathcal R(-D_i))$ are 
right ideals of the Noetherian ring $\mathcal A$, hence finitely generated.
We know also already that 
$\text{Hom}(\mathcal T, \mathcal R[1]) = \text{Hom}(\mathcal T, \mathcal R(-D_i)[1]) = 0$.

By the exact sequences in Lemma \ref{SES}, we can deduce inductively that 
$\text{Hom}(\mathcal T, \mathcal M[p])$ are finitely generated 
for all divisorial sheaves $\mathcal M$.

Finally let $\mathcal M$ be an arbitrary coherent module.
Then there is an exact sequence $0 \to \mathcal N \to \mathcal P \to \mathcal M \to 0$ where 
$\mathcal P$ is a finite direct sum of divisorial sheaves.
First it follows that $\text{Hom}(\mathcal T, \mathcal M[1])$ is finitely generated.
Then $\text{Hom}(\mathcal T, \mathcal N[1])$ is also finitely generated because $\mathcal M$ is arbitrary.
Then it follows that $\text{Hom}(\mathcal T, \mathcal M)$ is also finitely generated.
\end{proof}


Now we can define the following functors:
\[
\begin{split}
&\Phi: D^-(\text{coh}(\mathcal X)) \to D^-(\text{mod-}\mathcal A), \,\,\, \Phi(\bullet) 
= R\text{Hom}_{\mathcal X}(\mathcal T, \bullet), \\
&\Psi: D^-(\text{mod-}\mathcal A) \to D^-(\text{coh}(\mathcal X)), \,\,\, \Psi(\bullet) 
= \bullet \otimes_{\mathcal A}^{\mathbf L} \mathcal T.
\end{split}
\]

By the same argument as in \cite{TU} Lemma 3.3, we prove the following:

\begin{Thm}[\cite{Bondal}, \cite{Rickard}, \cite{TU}]
(1) $\Phi$ is an equivalence and $\Psi$ is its quasi-inverse.

(2) $\Phi$ induces an equivalence $D^b(\text{coh}(\mathcal X)) \to D^b(\text{mod-}\mathcal A)$.
\end{Thm}

\begin{proof}
This is the same as the proof of \cite{TU} Lemma 3.3.
We recall it briefly for the convenience of the reader.

(0) $\Psi$ is the left adjoint of $\Phi$:
\[
R\text{Hom}_{\mathcal X}(y \otimes_{\mathcal A}^{\mathbf L} \mathcal T, x) \cong 
R\text{Hom}_{\mathcal A}(y, R\text{Hom}_{\mathcal X}(\mathcal T,x)).
\]

(1) First we see that $\Phi\Psi \cong \text{Id}_{D^-(\text{mod-}\mathcal A)}$, i.e., 
\[
R\text{Hom}(\mathcal T, \bullet \otimes_{\mathcal A}^{\mathbf L} \mathcal T)
\cong \bullet.
\]
This is true for free modules, hence for all $D^-(\text{mod-}\mathcal A)$.

In order to prove that the adjunction morphism $\Psi\Phi \to \text{Id}_{D^-(\text{coh}(\mathcal X))}$ is also an isomorphism, 
we use the generator condition 
that $\Phi(c) \cong 0$ implies $c \cong 0$ for $c \in D^-(\text{coh}(\mathcal X))$.
From a distinguished triangle 
\[
\Psi\Phi(x) \to x \to c \to \Psi\Phi(x)[1]
\]
and $\Phi\Psi\Phi(x) \cong \Phi(x)$, we obtain $\Phi(c) \cong 0$. 

(2) First we have that $\Phi(D^b(\text{coh}(\mathcal X))) \subset D^b(\text{mod-}\mathcal A)$
by Lemma \ref{coherent}.
For an $A$-module $y$, we consider a natural morphism $\alpha: \tau_{< - m}\Psi(y) \to \Psi(y)$.
Since $\Phi$ is bounded, $\Phi(\alpha): \Phi(\tau_{< - m}\Psi(y)) \to \Phi\Psi(y) \cong y$ vanishes 
for large $m$. 
Hence $\alpha = 0$, and $\Psi(y)$ is bounded. 
\end{proof}


Graduate School of Mathematical Sciences, University of Tokyo,
Komaba, Meguro, Tokyo, 153-8914, Japan. 

kawamata@ms.u-tokyo.ac.jp

\end{document}